\documentclass[12pt]{amsart}
\usepackage[foot]{amsaddr}

\usepackage{amsmath, amssymb}

\usepackage[centering,text={15.5cm,22cm},marginparwidth=20mm]{geometry}

\usepackage[utf8]{inputenc}
\usepackage{enumerate}

\usepackage{hyperref}

\usepackage{lmodern}

\newtheorem{theorem}{Theorem}[section]
\newtheorem{lemma}[theorem]{Lemma}
\newtheorem{proposition}[theorem]{Proposition}
\newtheorem{corollary}[theorem]{Corollary}

\theoremstyle{definition}
\newtheorem{definition}[theorem]{Definition}
\newtheorem{remark}[theorem]{Remark}

\numberwithin{equation}{section}

\newcommand{\R}{\mathbb{R}}
\newcommand{\C}{\mathbb{C}}

\newcommand{\BF}{\mathbf{F}}
\newcommand{\BP}{\mathbf{P}}

\newcommand{\CA}{\mathcal{A}}

\newcommand{\CF}{\mathcal{F}}
\newcommand{\CM}{\mathcal{M}}
\newcommand{\CO}{\mathcal{O}}

\newcommand{\fk}{\mathfrak{k}}

\newcommand{\fg}{\mathfrak{g}}
\newcommand{\ft}{\mathfrak{t}}

\newcommand{\ev}[1]{\langle#1\rangle}
\newcommand{\coker}{\mathrm{coker}\! \ }
\newcommand{\vol}{\mathrm{vol}}
\newcommand{\Crit}{\mathrm{Crit}\ \!}
\newcommand{\jkres}{\mathrm{jkres}}
\newcommand{\Supp}{\mathrm{Supp}\! \  }

\begin{document}
%
\title{Localization for $K$-Contact Manifolds}


\author[L. Casselmann]{Lana Casselmann}
\email{lana.casselmann@uni-hamburg.de}
\author[J. M. Fisher]{Jonathan M. Fisher}
\email{jonathan.m.fisher@gmail.com}
\address{Fachbereich Mathematik, Universität Hamburg, Germany}

\begin{abstract}
We prove an analogue of the Atiyah-Bott-Berline-Vergne localization formula in the setting of equivariant basic cohomology of $K$-contact manifolds. 
As a consequence, we deduce analogues of Witten's nonabelian localization and the Jeffrey-Kirwan residue formula, which relate equivariant basic integrals on a contact manifold $M$ to basic integrals on the contact quotient $M_0 := \mu^{-1}(0)/G$, 
where $\mu$ denotes the contact moment map for the action of a torus $G$. In the special case that $M \to N$ is an equivariant Boothby-Wang fibration, our formulae reduce to the usual ones for the symplectic manifold $N$.
\end{abstract}

\maketitle

\tableofcontents

\section{Introduction}

Let $(M, \alpha)$ be a compact connected contact manifold of dimension $2n+1$. Then $M$ has a natural foliation $\CF$ whose leaves are the orbits of the Reeb vector field $R$. 
If $R$ integrates to a free $S^1$-action, then the space of leaves $M/\CF$ is naturally a symplectic manifold of dimension $2n$ and via the pullback of the projection, we can identify differential forms on $M/\CF$ with \emph{basic} differential forms $\Omega(M, \CF) \subset \Omega(M)$. Usually, however, $R$ does not integrate to an $S^1$-action and the space of leaves fails to be a manifold. 
Nevertheless, we can always consider the subcomplex $\Omega(M, \CF) \subset \Omega(M)$ of basic differential forms. 
The basic cohomology of $M$ is the cohomology of this complex, and it behaves very much like the cohomology of a compact $2n$-dimensional symplectic manifold, at least under the $K$-contact assumption. 
Suppose now that in addition a torus $G$ acts on $M$, preserving the contact form. Then using the Cartan model of equivariant cohomology, we obtain a subcomplex
$\Omega_G(M, \CF) \subseteq \Omega_G(M)$ of \emph{Reeb basic equivariant differential forms}. The corresponding cohomology ring $H_G(M, \CF)$ is a module over $H_G := H_G(\mathrm{point})$. 
In what follows, we denote by $\mu: M \to \fg^\ast$ the contact moment map defined by $\ev{\mu, \xi} := \alpha(\xi_M)$, by $\{\phi_t\}$ the flow of $R$ and by $T$ its closure. Our first result is an analogue of the Atiyah-Bott-Berline-Vergne localization formula \cite{AtiyahBottMomentMap,BerlineVergne}.

\begin{theorem} \label{thm-localization} Suppose a torus $G$ acts on a $K$-contact manifold $(M, \alpha)$ such that $G$ preserves $\alpha$, and suppose in addition that the $G$-fixed points have closed Reeb orbits. 
Then we have for all $\eta \in H_G(M, \CF)$ the identity
\begin{equation} \int_M \alpha \wedge \eta = \sum_{C_j \subseteq C} \int_{C_j} \frac{i_j^\ast (\alpha \wedge \eta)}{e_G(\nu C_j,\CF)}, \end{equation}
where $C = \Crit\mu$, $i_j: C_j \hookrightarrow M$ denotes the inclusion of the connected components $C_j \subseteq C$, and $e_G(\nu C_j,\CF)$ denotes the equivariant basic Euler class of the normal bundle to $C_j$.
\end{theorem}

\begin{remark} 
We note that for this result, it is sufficient to assume that all $G$-fixed points have a closed Reeb orbit, an assumption that is weaker than assuming 0 to be a regular value of $\Psi$ and that is automatically satisfied for total spaces in the Boothby-Wang fibration.

This theorem is closely related to results obtained in \cite{ToebenLoc, GNTlocalization}.
\end{remark}

Our second main theorem is an application of Theorem \ref{thm-localization} in the case that 0 is a regular value of the contact moment map $\mu$ to obtain an integration formula relating integration of equivariant basic forms on $M$ to integration of basic forms on $M_0 := \mu^{-1}(0) / G$, 
generalizing the results of Witten \cite{WittenNL} and Jeffrey-Kirwan \cite{JeffreyKirwan95} in the symplectic case. For any $\eta \in H_G(M, \CF)$, with $s= \dim G$, define a function $I^\eta(\epsilon)$ depending on a real parameter $\epsilon > 0$ by
\begin{equation} \label{eqn-ieta-defn}
  I^\eta(\epsilon) = \frac{1}{(2\pi i)^s \vol(G)} \int_{M \times \fg} \alpha \wedge \eta(\phi)  \wedge e^{id_G\alpha(\phi)-\epsilon |\phi|^2/2} d\phi.
\end{equation}
We denote by $\eta_0$ the image of $\eta$ under the natural basic Kirwan map (cf.~Theorem \ref{thm basic surj}) $H_G(M, \CF) \to H(M_0, \CF_0)$, and let $\alpha_0$ denote the quotient contact form on $M_0$.
\allowdisplaybreaks
\begin{theorem} \label{thm-integration} For any $\eta \in H_G(M, \CF)$, there exists some constant $c>0$ such that as $\epsilon \to 0^+$, $I^\eta(\epsilon)$ obeys the asymptotic
  \begin{equation} I^\eta(\epsilon) = \tfrac{1}{n_0}\int_{M_0} \alpha_0 \wedge \eta_0 \wedge  e^{\epsilon \Theta  + id\alpha_0} + o(\epsilon^{-s/2}e^{-c/ \epsilon}), \end{equation}
  where $\Theta \in  \!  H^4(M_0, \CF_0)$ is the class corresponding to $-\frac{<\phi,\phi>}{2}\in \! H_G^4(\mu^{-1} (0),\CF)$ $\simeq$ \linebreak $ H^4(M_0, \CF_0)$, with $n_0$ denoting the order of the kernel of the action of $G$ on $\mu^{-1}(0)$, that is, its regular isotropy.
\end{theorem}

A particular consequence of this theorem is the identity
\[\int_{M_0} \alpha_0 \wedge \eta_0 \wedge  e^{id\alpha_0} = n_0 \lim_{\epsilon \to 0^+} I^\eta(\epsilon),\]
which expresses intersection pairings on $M_0$ as limits of equivariant intersection pairings on $M$. 

The main ingredients in the proof of the Theorem \ref{thm-integration} are the result that the distribution $\BF(\int_M \alpha \wedge \eta \wedge e^{id_G\alpha})$, where $\BF$ denotes Fourier transformation, is piecewise polynomial and smooth near 0, and a particular expression for the polynomial this distribution coincides with near 0. With these properties and a result of Jeffrey-Kirwan, we then obtain the last of our main theorems.

\begin{theorem} \label{thm-residue} Let $\eta_0$ denote the image of $\eta \in H_G(M, \CF)$ under the Kirwan map. Then we have
\small
  \begin{equation} 
    \int\limits_{M_0} \alpha_0 \wedge \eta_0  \wedge e^{id\alpha_0} = \frac{n_0}{\vol(G)}
      \jkres\left( \sum_{C_j \subseteq C} e^{-i\ev{\mu(C_j), \phi}} \int\limits_{C_j} \frac{i_j^\ast \left(\alpha \wedge \eta(\phi)  \wedge e^{id\alpha} \right)}{e_G(\nu C_j,\CF)} [d\phi] \right). 
    \end{equation}
\normalsize
\end{theorem}

\begin{remark} 
In \S\ref{sec-boothby} we explain in detail how Theorems \ref{thm-localization} and \ref{thm-residue} may be used to deduce the analogous theorems for symplectic manifolds that occur as $M/\CF$ in the case that $R$ induces a free $S^1$-action. 
In this sense, these theorems provide a strict generalization of their symplectic analogues, at least in the case of an integral symplectic form and a Hamiltonian group action that lifts to the $S^1$-bundle in the Boothby-Wang fibration \cite{BoothbyWang}.
\end{remark}

\begin{remark} The first named author has obtained a surjectivity result for the basic contact Kirwan map \cite{Casselmann2016}. 
Since basic cohomology satisfies Poincar\'e duality (see Lemma \ref{lemma-poincare}), Theorem \ref{thm-residue} in principle provides a method to compute the kernel of the basic Kirwan map, and therefore allows one to compute the basic cohomology ring of the quotient.
\end{remark}

\subsection*{Acknowledgements}
We thank Oliver Goertsches and Lisa Jeffrey for helpful discussions and critical reading of the manuscript, and the anonymous referee for their careful reading and constructive comments on the paper. This research was conducted as part of GRK 1670 "Mathematics inspired by string theory and quantum field theory", funded by the Deutsche Forschungsgemeinschaft (DFG).

\section{\texorpdfstring{$K$}{K}-Contact Manifolds}\label{sec K contact}
\subsection{Contact Manifolds}

Let $M$ be a smooth connected $2n+1$-dimensional manifold. A \emph{contact form} on $M$ is a 1-form $\alpha$ such that $\alpha \wedge (d\alpha)^n$ is nowhere vanishing. A \emph{contact manifold}
is such a pair $(M, \alpha)$. 
Note that we take the contact form $\alpha$, and not just the induced hyperplane distribution $\ker \alpha$, as part of the data defining a contact manifold. On any such manifold there is a distinguished vector field, called the \emph{Reeb vector field} (which we usually denote by $R$), which is uniquely determined by the two conditions 
\begin{align*}
  \iota_R \alpha = 1, \quad   \iota_R d\alpha = 0. 
\end{align*}
Note that these conditions imply that $L_R \alpha = 0$. The contact form gives a direct sum decomposition $ T M = \ker \alpha \oplus \ev{R}$,
and we note that $\ker\alpha$ is a symplectic vector bundle over $M$ with symplectic form $d\alpha$.

\begin{definition} A \emph{contact metric} $g$ on $(M, \alpha)$ is a Riemannian metric $g$ on $M$, such that under the decomposition $TM \cong \ker\alpha \oplus \ev{R}$, we have
$g = g' \oplus (\alpha \otimes \alpha)$, where $g'$ is a $d\alpha$-compatible metric on $\ker \alpha$. We say that $(M, \alpha, g)$ is \emph{$K$-contact} if $g$ is a contact metric for which the Reeb vector field is Killing, i.e., such that $L_Rg =0$.
\end{definition}

The Reeb vector field $R$ generates a free $\mathbb{R}$-action on $M$ and induces a foliation $\CF$ on $M$. 
However, because the $\mathbb{R}$-action is usually not proper, the space of leaves $M / \CF$ can be badly behaved and is not necessarily a manifold. We work with compact $M$. The $K$-contact condition then implies the following, which is the main technical tool which allows us to overcome this difficulty.

Since $R$ is Killing, its flow $\phi_t$ generates a 1-parameter subgroup of the group of isometries of $(M,g)$. Since $M$ is compact, $\operatorname{Iso}(M, g)$ is a compact Lie group and, hence, the closure of $\phi_t$ in $\operatorname{Iso}(M, g)$ is a torus $T$. By construction, $R$ is the fundamental vector field of a topological generator of $T$. Since $\phi_t^\ast \alpha = \alpha$, it follows that $\alpha$ is preserved by all of $T$.

\subsection{Basic cohomology}

We define the \emph{$\CF$-basic} (or simply \emph{basic}) forms on $M$ to be
\begin{equation*} \Omega^k(M, \CF) = \{ \eta \in \Omega^k(M)  \ | \  L_X \eta = 0 = \iota_X \eta \ \forall X \in \mathfrak{X}(\CF)\}, \end{equation*}
where $\mathfrak{X}(\CF)$ denotes the vector fields which are tangent to the foliation.
Then one immediately sees that $d\Omega^k(M, \CF) \subseteq \Omega^{k+1}(M, \CF)$, so that $\Omega(M, \CF)$ is a subcomplex of the de Rham complex of $M$. Recall that $M$ is compact.
There is a natural Poincar\'e pairing on basic forms defined by
\begin{equation*} (\xi, \eta) \mapsto \int_M \alpha \wedge \xi \wedge \eta.
\end{equation*}

\begin{definition} The \emph{basic cohomology} of $(M, \CF)$, denoted by $H^*(M, \CF)$, is the cohomology of the complex $\Omega(M, \CF)$.
\end{definition}

\begin{lemma} \label{lemma-poincare}
The Poincar\'e pairing descends to a well-defined pairing on basic cohomology. If $M$ is a compact $K$-contact manifold, then the basic cohomology groups are finite-dimensional, $H^r(M, \CF)=0$ for $r>2n$ and the Poincar\'e pairing is non-degenerate. 
\end{lemma}
\begin{proof} 
See, e.g., \cite[Proposition~7.2.3]{boyer2008sasakian}.
\end{proof}

\subsection{Equivariant basic cohomology}
We now define equivariant basic cohomology and give its basic properties. See \cite{GNTequivariant, ToebenLoc, GNTlocalization, Casselmann2016} for related constructions. We suppose now that a torus $G$ acts on $M$, preserving the contact form.
Then the action of $G$ commutes with the flow of the Reeb vector field. In particular, the action of $G$ preserves the foliation $\CF$ and commutes with the $T$-action.

Let $\fg$ denote the Lie algebra of $G$. For any $\xi \in \fg$, we let $\xi_M$ denote the corresponding fundamental vector field on $M$, defined as
\[ \xi_M(x) = \left.\tfrac{d}{dt}\exp(t\xi) \cdot x \right|_{t=0}. \] 
Let $S(\fg^\ast)$ denote the symmetric algebra of $\fg^\ast$ (i.e. $S(\fg^\ast)$ is the algebra of polynomial functions on the vector space $\fg$). The complex of equivariant differential forms $\Omega_G(M)$ is the complex with underlying vector space
\begin{equation*} \Omega_G(M) = \left(S(\fg^\ast)\otimes \Omega(M) \right)^G,  \end{equation*}
with grading $\deg(f \otimes \eta) = 2 \deg(f) + \deg(\eta)$ and differential
\begin{equation*} d_G(f \otimes \eta)(\xi) = (d\eta) f(\xi) - (\iota_{\xi_M} \eta)f(\xi). \end{equation*}
(Note the sign in the differential -- this is chosen to be consistent with \cite{JeffreyKirwan95}.)

\begin{definition} The complex $\Omega_G(M, \CF)= \left(S(\fg^\ast)\otimes \Omega(M,\CF) \right)^G$ of \emph{equivariant basic forms} is the subcomplex of $\Omega_G(M)$ consisting of basic equivariant differential forms. The \emph{equivariant basic cohomology of $M$} is the cohomology of this subcomplex, denoted by $H_G(M, \CF)$. 
We also denote by $\Omega_{G,c}(M), H_{G,c}(M)$, etc. the complex of compactly supported equivariant differential forms, classes, etc.
\end{definition}

\begin{remark} The complexes $\Omega_G(M, \CF)$, $\Omega_G(M)$ and their cohomologies $H_G(M, \CF)$, $H_G(M)$ are all naturally modules over $H_G:=$ $H_G(\mathrm{point})$ $\cong S(\fg^\ast)$.
\end{remark}

\begin{remark} More generally, one can define (equivariant) basic cohomology on the category of pairs $(M, \CF_M)$ consisting of a manifold $M$ with regular foliation $\CF_M$, (acted upon by $G$ s.t. $\Omega(M, \CF_M)$ is a $G^*$-algebra (cf.~\cite[Definition~2.3.1]{GS99})), 
and morphisms $(M, \CF_M) \to (N, \CF_N)$ given by (equivariant) foliation-preserving smooth maps, i.e. smooth maps which take leaves to leaves. 
In particular, the $H_G$-module structure on $H_G(M, \CF)$ is induced by the pullback of the map projecting $M$ to the 1-point manifold with trivial foliation.
\end{remark}

\begin{lemma}[{\cite[Proposition~10]{Casselmann2016}}]\label{lemma if homotopy equiv then cohom isom}
	Let $A,B\subset M$ be $G\times \{\phi_t\}$-in\-va\-ri\-ant
	submanifolds and assume that we have $G\times \{\phi_t\}$-equivariant homotopy inverses $f:A\to B$ and $g:B\to A$. Then $f^*: H_G(B, \CF)\to H_G(A, \CF)$ is an isomorphism with inverse $g^*$.
\end{lemma}
The group $G\times \{\phi_t\}$ is in general non-compact, which complicates finding, e.g., invariant objects or tubular neighborhoods. As mentioned above, the tool to overcome this obstacle is considering the closure $T$ of $\{\phi_t\}$, in particular, we often consider the action of the torus $G\times T$. 
A closed $G\times \{\phi_t\}$-invariant submanifold $A\subset M$ is automatically $G\times T$-invariant, hence, there exist arbitrarily small $G\times T$-invariant tubular neighborhoods that retract onto $A$. 
These retractions are, in particular, $G\times \{\phi_t\}$-equivariant. Lemma \ref{lemma  if homotopy equiv then cohom isom} and the corresponding well known statement in ordinary equivariant cohomology then yield

\begin{lemma} \label{lemma-bundle-cohomology}
Let $i: A \hookrightarrow M$ be the inclusion of a $G \times T$-invariant submanifold, and let $U$ be a $G \times T$-invariant tubular neighborhood of $A$ in $M$. Let $p: U \to A$ denote the projection map. 
Then $i^*: H_G(U)\to H_G(A)$ and  $i^*: H_G(U,\CF)\to H_G(A,\CF)$ are isomorphisms with inverse $p^*$.
\end{lemma}

\begin{definition}\label{def rel cohom} Let $A \subseteq M$ be a $G \times T$-invariant closed submanifold of $M$. We define the complex $\Omega_G(M, A, \CF)$ to be the kernel of the pullback $\Omega_G(M, \CF) \to \Omega_G(A, \CF)$. 
Since the pullback commutes with the differential, $\Omega_G(M, A, \CF)$ is a differential subcomplex of $\Omega_G(M, \CF)$. We denote its cohomology by $H_G(M, A,\CF)$.
\end{definition}
\begin{proposition}\label{prop LES cohom} There is a natural long exact sequence in equivariant basic cohomology
\[ \cdots \to H_G^k(M, A,\CF) \to H_G^k(M,\CF) \to H_G^k(A,\CF) \to \cdots \]
\end{proposition}
\begin{proof} By standard homological algebra, this follows from the existence of the short exact sequence $ 0 \to \Omega_G(M, A, \CF) \to \Omega_G(M, \CF) \to \Omega_G(A, \CF) \to 0$.
\end{proof}

\begin{proposition}\label{prop rel cohom iso comp supp} We have an isomorphism 
$H_G(M, A,\CF) \cong  H_{G,c}(M \setminus  A, \CF)$.
\end{proposition}
\begin{proof} We follow the same line of arguments as in the usual equivariant case (see \cite[Theorem~11.1.1]{GS99}). Extending by $0$ gives a natural inclusion of equivariant basic forms
$ \Phi: \Omega_{G,c}(M \setminus A, \CF) \to \Omega_G(M, A, \CF)$.  
$\Phi$ induces an isomorphism on cohomology:
First, let $i:A\hookrightarrow U$ be a $G \times T$-invariant tubular neighborhood of $A$ and let $\eta \in \Omega_G(M, A, \CF)$ be an equivariantly closed form. 
Then by Lemma \ref{lemma-bundle-cohomology}, we can find $\omega \in \Omega_G(U,\CF)$ so that $\eta|_U = d_G\omega$. Then $i^*\omega$ is equivariantly closed, so $\lambda:=\omega-\pi^*i^*\omega$ satisfies $\lambda \in \Omega_G(U, A, \CF)$ and $\eta|_U = d_G\lambda$. 
Let $\rho$ be a $G\times T$-invariant smooth function which is identically 1 on some smaller neighborhood of $A$ and which is compactly supported in $U$. Then $\eta - d_G(\rho\lambda)\in \Omega_{G,c}(M \setminus A, \CF)$. This shows surjectivity. 
Now suppose that $\eta \in \Omega_{G,c}(M \setminus A, \CF)$ is in the kernel of the induced map on cohomology, i.e., that there exists $\lambda \in \Omega_G(M, A, \CF)$ such that $\eta = d_G\lambda$. 
Then since $\eta$ is compactly supported on $M \setminus A$, there exists a neighborhood $U$ of $A$ on which $\eta$ is identically zero. Therefore $\lambda$ is closed on $U$.  
Since $i^\ast \lambda =0$ by assumption, by Lemma \ref{lemma-bundle-cohomology}, as above, we have $\lambda = d_G\beta$ for some $\beta \in \Omega_G(U, A, \CF)$. 
Now let $\rho$ be an invariant smooth function which is identically 1 on a neighborhood of $A$ and which has compact support in $U$. 
Then $\widetilde\lambda := \lambda - d_G(\rho \beta) \in \Omega_{G,c}(M \setminus A, \CF)$ and we have $\eta = d_G\widetilde\lambda$. This shows injectivity.
\end{proof}

\subsection{The contact moment map}
Recall that for any $\xi \in \fg$, we let $\xi_M$ denote the corresponding fundamental vector field on $M$.
By the invariance of the contact form $\alpha$, we have
$ 0 = L_{\xi_M} \alpha = d(\iota_{\xi_M}\alpha) + \iota_{\xi_M} d\alpha$. 
Then the contact moment map for the $G$-action on $(M, \alpha)$ is the function $\mu: M \to \fg^\ast$ defined by
\[ \ev{\mu, \xi} = \alpha(\xi_M). \]
\begin{proposition}[{\cite[Lemmata 7 and 9]{Casselmann2016}}]\label{prop crit mu} Suppose that $0$ is a regular value of $\mu$. Then $\Crit \mu$ is the union of all 1-dimensional $G \times T$-orbits, and each connected component of $\Crit \mu$ is a closed submanifold of $M$ of even codimension.
\end{proposition}

If $0$ is a regular value of $\mu$, the level set $\mu^{-1}(0)$ is a smooth $G\times T$-invariant submanifold of $M$, on which $G$ acts locally freely. 
We define the contact reduction $M_0 := \mu^{-1}(0) / G$, which is a contact orbifold (and an honest manifold if the action of $G$ on $\mu^{-1}(0)$ is free). 
Since $G$ and the Reeb flow commute and the Reeb orbits are transversal to the $G$-orbits along $\mu^{-1}(0)$, $\Omega(\mu^{-1}(0),\CF)$ is a $\fg$-dga of type (C) (cf.~\cite[Def. 2.3.4]{GS99}) and, hence, we have $H_G(\mu^{-1}(0), \CF) \cong H(\Omega(\mu^{-1}(0),\CF)_{\text{bas } \fg})$ (cf.~\cite[\S~5.1]{GS99} and \cite[Proof of Lemma 3.18]{GT2016equivariant}) via the Cartan map. 
This implies that we have an isomorphism $H_G(\mu^{-1}(0), \CF)\cong H(M_0, \CF_0)$, where 
the later denotes the cohomology of the $R$-basic differential forms on $M_0$. 
There is a natural map
\[ \kappa: H_G(M, \CF) \to H_G(\mu^{-1}(0), \CF) \cong H(M_0, \CF_0), \]
induced by the inclusion $\mu^{-1}(0)\subset M$, which we call the \emph{basic Kirwan map}, or simply the \emph{Kirwan map} when its meaning is clear from context.

\begin{theorem}[{\cite[Theorem~2]{Casselmann2016}}]\label{thm basic surj} If $0$ is a regular value of $\mu$, then the basic Kirwan map $H_G^\ast(M, \CF) \to H_G^\ast(\mu^{-1}(0), \CF)$ is surjective.
\end{theorem}

We will need a local normal form of the contact moment map in a neighborhood of $\mu^{-1}(0)$. In order to obtain it, we need to show the uniqueness of certain coisotropic embeddings into contact manifolds. 
To this end, we first prove an equivariant contact Darboux Theorem for submanifolds. 
Note that while a contact Darboux Theorem for contact forms in a neighborhood of a point (see, e.g., \cite[Theorem~2.24]{geiges2006}) is well-known, a contact Darboux Theorem for neighborhoods of submanifolds exists, to our knowledge, so far only for contact structures (\cite[Theorem 3.6]{lerman2002contact}) or submanifolds to which the Reeb vector fields are nowhere tangent (\cite[Theorem~B]{arnold4symplectic}). 
We follow Lerman's approach for contact structures. Note that his proof does not generally work for contact \emph{forms} because his function $g_t$ (which is $\varphi_t^*(\dot \alpha_t(R_t))$ in the notation of the upcoming proof) might not vanish. It is, however, applicable in our case, because we make the additional assumption that the Reeb vector fields coincide on a neighborhood of the submanifold.
\begin{theorem}[Equivariant contact Darboux Theorem]\label{thm Darboux}
	Let $Y$ be a closed submanifold of $X$ and let $\alpha^0$ and $\alpha^1$ be two contact forms on $X$ with Reeb vector fields $R_i$, $i=0,1$. Suppose that $\alpha^0_x = \alpha^1_x$ and $d\alpha^0_x = d\alpha^1_x$ for every $x \in Y$ and that there is a neighborhood $U$ of $Y$ in $X$ such that $R_0=R_1$ on $U$. Then there exist neighborhoods $U_0, U_1$ of $Y$ in $X$ and a diffeomorphism $\varphi: U_0 \rightarrow U_1$ such that $\varphi|_Y = \operatorname{id}|_Y$ and $\varphi^*\alpha^1 = \alpha^0$. 
	
	Moreover, if a compact Lie group $K$ acts on $X$, preserving $Y$, $U$, and the two contact forms $\alpha^0, \alpha^1$, then we can choose $U_0$ and $U_1$ $K$-invariant and $\varphi$ $K$-equivariant.
\end{theorem}
\begin{proof}
	Consider the family of 1-forms $\alpha^t := t \alpha^1 + (1-t)\alpha^0$, $t \in [0,1]$. For every $x \in Y$ and every $t \in [0,1]$, we have $\alpha^t_x = \alpha^1_x=\alpha^0_x$ and $d \alpha^t_x = d \alpha^1_x=d\alpha^0_x$. 
	It follows that $\alpha^t$ are contact forms in a neighborhood of $Y$ for every $t \in [0,1]$: 
	By maximality of the degree, there is a smooth function $f: X \times [0,1] \to \mathbb{R}$ such that $\alpha_t \wedge (d\alpha_t)^n = f\alpha_0 \wedge (d\alpha_0)^n$. $f^{-1}(\mathbb{R}\setminus \{0\})$ is 
	open and contains $Y\times [0,1]$, so for every $(x,t)\in Y \times [0,1]$, there exists a neighborhood $U(x,t)$ of the form $U_t(x) \times (t-\epsilon_{x,t}, t+ \epsilon_{x,t})\cap [0,1]$, $\epsilon_{x,t}>0$ such that $f|_{U(x,t)}\neq 0$. 
	Since $[0,1]$ is compact, there are $t_1, ..., t_N$: $[0,1] = \cup_{i=1}^N(t_i-\epsilon_{x,t_i}, t_i+ \epsilon_{x,t_i})\cap [0,1]$. 
	Then $\widetilde U :=\cup_{x\in Y}\left( \cap_{i=1}^N U_{t_i}(x)\right)$ is open, contains $Y$ and  $f$ does not vanish on $\widetilde U \times [0,1]$. 
	Thus, all $\alpha_t$ are contact forms on $\widetilde U$. W.l.o.g., we assume that they are contact forms at least on all of $U$. $\alpha^t$ are $K$-invariant because $\alpha^0$ and $\alpha^1$ are. Let $R_t$ denote the Reeb vector field of $\alpha^t$. Since $R_t$ is uniquely determined, 
	$R_t$ is also $K$-invariant and, on $U$, we have $R_t=R_0$. 
	Set 
	\[\dot \alpha_t := \frac{d}{dt}\alpha^t=\alpha_1-\alpha_0.\]
	$\dot \alpha_t$ vanishes on $Y$ and, on $U$, it is $\dot \alpha_t(R_0)=0$. 
	Define a $K$-invariant time dependent vector field $X_t$ tangent to the contact distribution $\xi_t:=\ker \alpha_t$ and vanishing on $Y$ by
	\[ X_t := \left( d \alpha_t|_{\xi_t}\right)^{-1}(-\dot \alpha_t|_{\xi_t}). \]
	Then we have $(\iota_{X_t}d\alpha_t)|_{\xi_t} = -\dot \alpha_t|_{\xi_t}=(\dot \alpha_t(R_t)\alpha_t - \dot \alpha_t)|_{\xi_t}$ and $(\iota_{X_t}d\alpha_t)(R_t)=0=(\dot \alpha_t(R_t)\alpha_t - \dot \alpha_t)(R_t)$. Hence, since $X_t \in \xi_t$, 
	\[L_{X_t}\alpha_t =\iota_{X_t}d\alpha_t = \dot \alpha_t(R_t)\alpha_t - \dot \alpha_t .\]
	Denote the time dependent flow of $X_t$ by $\varphi_t$. $\varphi_t$ is defined on a neighborhood $V$ of $Y$ since $X_t$ vanishes on $Y$, $K$-invariant because $X_t$ is $K$-invariant, and $\varphi_t|_Y=\mathrm{id}_Y$. Then 
	\[\frac{d}{dt}(\varphi_t^*\alpha_t)=\varphi_t^*(L_{X_t}\alpha_t + \dot \alpha_t) = \varphi_t^*(\dot \alpha_t(R_t)\alpha_t).\]
	On $U$, $0=\dot \alpha_t(R_0)=\dot \alpha_t(R_t)$. 
	We will find a small neighborhood $U_0$ of $Y$ with $\varphi_t(U_0) \subset U$ for every $t$, then we have $\frac{d}{dt}(\varphi_t^*\alpha_t)=0$ on $U_0$ and, hence, $\varphi_t^*\alpha_t\equiv \varphi_0^*\alpha_0 = \alpha_0$. $\varphi_1 :U_0 \rightarrow \varphi_1(U_0)=:U_1$ hence defines the desired $K$-invariant contactomorphism. 
	To find $U_0$, note that for every $(x,t)\in Y\times [0,1]$, there exists a neighborhood $U(x,t)$ of the form $U_t(x) \times (t-\epsilon_{x,t}, t+ \epsilon_{x,t})\cap [0,1]$, $\epsilon_{x,t}>0$ such that $\varphi (U(x,t)) \subset U$. 
	Since $[0,1]$ is compact, there are $t_1, ..., t_N$: $[0,1] $ $= \cup_{i=1}^N(t_i-\epsilon_{x,t_i}, t_i+ \epsilon_{x,t_i})\cap [0,1]$. 
	Then $U_0 :=\cup_{x\in Y}\left( \cap_{i=1}^N U_{t_i}(x)\right)$ is open, contains $Y$ and $\varphi( U_0 \times [0,1])\subset U$.
\end{proof}

\begin{theorem}[Contact Coisotropic Embedding Theorem]\label{theorem contact coiso emb}
	Let $\alpha$ be a 1-form on a manifold $Z$ such that $d \alpha$ is of constant rank. Suppose that a compact Lie group $K$ acts on $Z$, leaving $\alpha$ invariant. Suppose that there are two 
	contact $K$-manifolds $(X_1, \alpha_1), (X_2, \alpha_2)$ and $K$-equivariant embeddings $i_j: Z \rightarrow X_j$ such that
	\begin{enumerate}[(i)]
		\item $di_j(TZ) \cap \ker \alpha_j$ is coisotropic in $(\ker \alpha_j, d\alpha_j|_{\ker \alpha_j})$,
		\item $i_j^* \alpha_j = \alpha$ and $K$ preserves $\alpha_j$,
		\item there is a nowhere vanishing $K$-fundamental vector field $X_Z$ on $Z$, generated by $X\in \mathfrak{k}$, such that $di_j(X_Z)= R_j$, where $R_j$ denotes the Reeb vector field on $X_j$, and $R_j$ is the fundamental vector field generated by $X$ on all of $X_j$.   
		(In particular, the Reeb flow corresponds to the action of a subgroup of $K$ on $X_j$).
	\end{enumerate}
	Then there exist $K$-invariant neighborhoods $U_j$ of $i_j(Z)$ in $X_j$ and a $K$-equivariant diffeomorphism $\varphi : U_1 \rightarrow U_2$ such that $\varphi ^* \alpha_2 = \alpha_1$ and $i_2 = \varphi \circ i_1$.
\end{theorem}
	To prove this Theorem, we adjust the proof of the well-known Coisotropic Embedding Theorem for symplectic manifolds (see, e.g., \cite[Theorem 39.2]{GS84}) to the contact setting and extend it in order to obtain an equality of contact forms, not only of their differentials. 
	The following notation is used. $\xi_j := \ker \alpha_j$, $\zeta_j := di_j(TZ) \cap \ker \xi_j$, $\omega_j := d\alpha_j |_{\xi_j}$, $\perp:= \perp_{d\alpha}$, $\perp_j:= \perp_{\omega_j}$. 
	Note that by our assumptions, $\zeta_j$ is $K$-invariant and $\R R_j \subset di_j(TZ)$ and, hence, $di_j(TZ) = \zeta_j \oplus \R R_j$.

\begin{lemma}\label{lem vb iso normal bundle}
	$N_j := TX_j/di_j(TZ) \ \simeq \ (TZ^\perp/\R X_Z)^*$ as $K$-vector bundles over $Z$.
\end{lemma}
\begin{proof}
	Consider the maps 
	\[\begin{array}{rrcl}
		\varphi_j:& TX_j/di_j(TZ) &\to &(di_j(TZ^\perp)/\R R_j)^*\\
		&[v] &\mapsto &d\alpha_j(v, \cdot)|_{di_j(TZ^\perp)/\R R_j}.
	\end{array}\]
	Since $R_j \in \ker d \alpha_j$ and $di_j(TZ)\perp_{d\alpha_j} di_j(TZ^\perp)$, the map $\varphi_j$ is well-defined. By assumption, $di_j(TZ)\cap \xi_j$ is coisotropic. It follows that $di_j(TZ^\perp)^{\perp_{d\alpha_j}} \subseteq di_j(TZ)$. This, however, yields that $\varphi_j$ is injective. For dimensional reasons, $\varphi_j$ then has to be surjective, as well. Since $i_j$ is an equivariant embedding, we have $K$-equivariant isomorphisms $TZ^\perp/\R X_Z \simeq di_j(TZ^\perp)/\R R_j$.
\end{proof}
\begin{proof}[Proof of the Embedding Theorem]
	Realize $N_j$ as a $K$-invariant complement of $di_j(TZ)$ in $TX_j$ such that $\xi_j = \zeta_j \oplus N_j$. This is possible since $\R R_j \subset di_j(TZ)$. 
	By Lemma \ref{lem vb iso normal bundle}, we have a canonical $K$-equivariant vector bundle isomorphism $A: N_1 \rightarrow N_2$. Then for $v \in N_1$, $Av \in N_2$ is defined via 
	\[ \omega_1(v, di_1(w))=\omega_2(Av, di_2(w)) \text{ for every } di_j(w) \in di_j(TZ^\perp)\cap \zeta_j.\]	
	(A neighborhood of the zero section of) $N_j$ can be identified with a $K$-invariant tubular neighborhood $U_j$ of $i_j(Z)$ in $X_j$ via the exponential maps of $K$-invariant Riemannian metrics, where $Z$ embeds as the zero section. Then $A$ yields a $K$-equivariant diffeomorphism $\tilde A: U_1 \rightarrow U_2$ with $i_2=\tilde A \circ i_1$. 
	Set $\tilde \alpha_1 := \tilde A^*\alpha_2$. 
	Then $\tilde \alpha_1$ is a contact form on $U_1$. We want to apply Theorem \ref{thm Darboux}. $i_2=\tilde A \circ i_1$ implies that $i_1^*\alpha_1 = \alpha = i_2^*\alpha_2 = i_1^* \tilde \alpha_1$. 
	Hence, we have $(\tilde \alpha_1)_{i_1(z)}|_{di_1(TZ)}=(\alpha_1)_{i_1(z)}|_{di_1(TZ)}$. Furthermore, we have $d \tilde A |_{N_1}=A$ by construction, so $d\tilde A|_{N_1}:N_1 \subset \xi_1 \rightarrow N_2 \subset \xi_2$, which yields $(\tilde \alpha_1)_{i_1(z)}|_{\xi_1}=0=(\alpha_1)_{i_1(z)}|_{\xi_1}$. 
	Thus,  $(\tilde \alpha_1)_{i_1(z)}=(\alpha_1)_{i_1(z)}$ on all of $TX_1$. Since the Reeb vector fields are fundamental vector fields of the same element of $\mathfrak{k}$ and since $\tilde A $ is $K$-invariant, $d\tilde A(R_1(p))=R_2(\tilde A(p))$. 
	It follows that $\tilde \alpha_1(R_1)=1$ and $\iota_{R_1} d\tilde \alpha_1=0$, so $R_1$ is the Reeb vector field of $\tilde \alpha_1 $ on $U_1$. It remains to show that $(d\tilde \alpha_1)_{i_1(z)}=(d\alpha_1)_{i_1(z)}$ on $\xi_1 \times \xi_1$, which is seen as in the symplectic case.
	By Theorem \ref{thm Darboux}, there is a neighborhood $U$ of $i_1(Z)$ and a $K$-equivariant diffeomorphism $g$ of $U$ into $X_1$ s.t. $g|_{i_1(Z)}=\mathrm{id}_{i_1(Z)}$ and $g^*\tilde \alpha_1 = \alpha_1$. 
	Then $\varphi := \tilde A \circ g$, restricted to a small enough neighborhood, satisfies $\varphi^* \alpha_2 = \alpha_1$.
\end{proof}

\begin{lemma}\label{lem mu - M coisotropic}
	Suppose that 0 is a regular value of the contact moment map. Then the natural embedding $\mu^{-1}(0) \hookrightarrow M$ 
	satisfies \it{(i)-(iii)} of Theorem \ref{theorem contact coiso emb} with $K=G\times T$.
\end{lemma}
\begin{proof}
	 \it{(ii)} and \it{(iii)} are obviously satisfied. To show that the distribution $\zeta := T\mu^{-1}(0) \cap \ker \alpha$ is coisotropic in $(\ker \alpha, d \alpha|_{\ker \alpha}=:\omega)$, recall that 0 is a regular value of $\mu$, hence, 
		\begin{align}\label{eq Tp mu = ker d mu}
			T_p\mu^{-1}(0)=\ker d\mu_p.
		\end{align}
		$v \in \ker d\mu_p$ if and only if $ d\mu^X_p(v)=(d\iota_{X_M}\alpha)_p(v)=0 $ for every $X\in \fg$. Since $\alpha$ is $G$-invariant, $L_{X_M}\alpha=0$, and Cartan's formula yields that $v \in \ker d\mu_p$ if and only if  $d\alpha_p(X_M,v)=0$ for every $X \in \fg$. It follows that 
		\begin{align}\label{eq ker d mu}
			\ker d \mu_p = (T_p G \cdot p)^{\perp_{d\alpha}}
		\end{align}
		since the tangent space to the $G$-orbit consists of all fundamental vector fields. For $p \in \mu^{-1}(0)$, it is $0=\mu(p)(X) =  \alpha_p(X_{M}(p))$ for every $X \in \fg$. In particular, $T_p(G\cdot p) \subset \ker \alpha_p$. 
		It follows that $(T_p G \cdot p)^{\perp_{d\alpha}} =  (T_p G \cdot p)^{\perp_{\omega}}\oplus \R R_p$. Equations \eqref{eq ker d mu} and \eqref{eq Tp mu = ker d mu} yield $T_p \mu^{-1}(0) \cap \ker \alpha_p = T_p(G\cdot p)^{\perp_\omega} =: \zeta_p$. 
		Then $\zeta_p^{\perp_\omega}=T_p(G\cdot p)$. $\mu$ is $G$-invariant, so for every $X \in \fg$, $d\mu_p(X_{M}(p))=0$. 
		We obtain $\zeta_p^{\perp_\omega} = T_p(G\cdot p) \subset \ker d\mu_p = (T_p G \cdot p)^{\perp_{d\alpha}} $ and, hence, 
		$ \zeta_p^{\perp_\omega}\subset (T_p G \cdot p)^{\perp_{d\alpha}}\cap \ker \alpha_p = (T_p G \cdot p)^{\perp_{\omega}}=\zeta_p$,
		$\zeta$ is coisotropic. 
\end{proof}

\begin{lemma}\label{lem mu - mu times g coisotropic}
	Suppose that 0 is a regular value of the contact moment map. Then the embedding $\mu^{-1}(0) \cong \mu^{-1}(0) \times \{0\} \hookrightarrow \mu^{-1}(0) \times \fg^*$ satisfies \it{(i)-(iii)} of Theorem \ref{theorem contact coiso emb} with $K=G\times T$, 
	where a neighborhood  $U = \mu^{-1}(0) \times V$ of $\mu^{-1}(0) \times \{0\} \subset \mu^{-1}(0) \times \fg^*$ is endowed with the contact form $\tilde \alpha := i^* \alpha + z(\theta)$, we denote the inclusion $\mu^{-1}(0) \hookrightarrow M$ by $i$, the coordinates on $\fg^*$ by $z$ and $\theta$ is a $G$-invariant $R$-basic connection form on $\mu^{-1}(0) \rightarrow \mu^{-1}(0)/G$. 
	Furthermore, $R$ is the Reeb vector field of $(\mu^{-1}(0) \times \fg^*,\tilde \alpha)$.
\end{lemma}
\begin{remark}\label{rem existence basic con form}
	Note that a $G$-invariant $R$-basic connection form has to exist: By \cite[Proposition~2.8]{Molino}, there always exists a connection that is \emph{adapted} to the lifted foliation, i.e., such that the tangent spaces to the leaves are horizontal. 
	Since $G\times T$ is compact, we can obtain a $G\times T$-invariant adapted connection form by averaging over the group. But this connection form then has to be basic, or, as Molino calls it, \emph{projectable}. 
\end{remark}

\begin{proof}
	Let $j: \mu^{-1}(0) \to \mu^{-1}(0) \times \fg^\ast$ denote the embedding given by $x \mapsto (x,0)$. Then $j^\ast\widetilde\alpha = i^\ast \alpha$ by construction. Choose an orthonormal basis $(X_i)$ of $\fg$ and denote its dual basis by $(u_i)$. Then $\theta = \sum \theta_i X_i$ and $z = \sum z_i u_i$. 
	With $\Omega^s = \theta_i \wedge... \wedge \theta_s$ and $dz^s = dz_i \wedge ... \wedge dz_s$, at $z=0$, we have
	\[ \widetilde\alpha \wedge (d\widetilde \alpha)^n = (-1)^{s(s+1)/2}s!\ i^\ast(\alpha \wedge (d\alpha)^{n-s}) \wedge \Omega^s \wedge dz^s, \]
	which is non-degenerate.  Therefore, there is a neighborhood $U = \mu^{-1}(0) \times V$ of $\mu^{-1}(0) \times \{0\}$ in $\mu^{-1}(0) \times \fg^\ast$ on which $\widetilde\alpha$ is a contact form. 
	$\theta$ is $R$-basic, so $\iota_R\theta = 0$ and $\iota_R \tilde \alpha = \iota_R i^*\alpha = i^* \iota_R \alpha = 1$. 
	$d\theta$ is $R$-basic, as well, so $\iota_R d\theta = 0$. $R$ is tangent to $\mu^{-1}(0)$, so $dz_i(R)=0$. We obtain $\iota_R d \tilde \alpha = \iota_R(i^* d \alpha + dz(\theta) + z(d\theta)) = 0$. By uniqueness, $R$ is the Reeb vector field of $(U,\tilde \alpha)$.
	It remains to show that the distribution $\zeta_p := T_p\mu^{-1}(0) \cap \ker \tilde \alpha_p$ is coisotropic in the symplectic vector bundle $(\ker \tilde \alpha, d \tilde \alpha|_{\ker \tilde \alpha}=:\omega)$. 
	The contact moment map $\tilde \mu$ on $(\mu^{-1}(0) \times \fg^*,\tilde \alpha)$ is easily computed to be $\tilde \mu(p,z)= z$, hence, $\tilde \mu^{-1}(0) = \mu^{-1}(0) \times \{0\} = i(\mu^{-1}(0))$. $d \tilde \mu = dz$ has $0$ as a regular value, so $\ker d \tilde \mu_{(p,0)}=T_{(p,0)}(\mu^{-1}(0)\times \{0\})$. 
	The rest of the proof works analogously to that of Lem\-ma \ref{lem mu - M coisotropic}.
\end{proof}

Applying Theorem \ref{theorem contact coiso emb} to the two coisotropic embeddings in Lemmata \ref{lem mu - M coisotropic} and \ref{lem mu - mu times g coisotropic}, we obtain a local normal form of $\mu$ around $\mu^{-1}(0)$.

\begin{proposition} \label{prop-normal-form}
Suppose that $0$ is a regular value of $\mu$. 
Then there is a $G\times T$-invariant neighborhood $U$ of $\mu^{-1}(0)$ which is equivariantly diffeomorphic to a neighborhood of $\mu^{-1}(0) \times \{0\}$ in $\mu^{-1}(0) \times \fg^\ast$ of the form $\mu^{-1}(0)\times B_h$, $B_h=\{z \in \fg \mid |z| \leq h\}$, such that in this neighborhood the contact form $\alpha$ is equal to $q^\ast \alpha_0 + z(\theta)$, where $\theta \in \Omega^1(\mu^{-1}(0), \CF, \fg)$ is a $G$-invariant, $\CF$-basic connection 1-form on $q: \mu^{-1}(0) \to \mu^{-1}(0) / G$. 
In particular, on $U$, the moment map is given by $\mu(p,z)=z$.
\end{proposition}

\section{Localization}
\subsection{Basic equivariant Thom isomorphism} \label{sec-thom}
Let $i: A \hookrightarrow M$ denote the inclusion of a $G\times T$-invariant closed submanifold of codimension $d$. 
The goal of this section is to construct a basic equivariant pushforward $i_\ast: H_G(A, \CF) \to H_G(M, \CF)$ which raises cohomological degree by $d$. We will follow the presentation in \cite[Chapter 10]{GS99} very closely. 

To begin, let $p: U \to A$ denote the projection of a $G \times T$-invariant tubular neighborhood. 
Since $U$ is a $G \times T$-equivariant fiber bundle over $A$, there is a well-defined pushforward map $p_\ast: \Omega^{k}_{G,c}(U) \to \Omega^{k-d}_G(A),$ defined by fiberwise integration. 
Note that $p_*$ maps equivariant basic forms to equivariant basic forms. From the definition of $p_\ast$ we immediately obtain the following, which shows that $p_\ast$ descends to a well-defined map on equivariant (basic) cohomology.

\begin{lemma}\label{lemma-pushforward} Let $p: U \to A$ be the projection and let $p_\ast: \Omega_{G,c}(U) \to \Omega_G(A)$ denote fiberwise integration. Then we have for all $\eta \in \Omega_{G,c}(U)$ and for all $\beta \in \Omega_G(A)$
\begin{equation*} \int_U p^\ast \beta \wedge \eta = \int_A \beta \wedge p_\ast \eta. \end{equation*}
\end{lemma}

The basic equivariant pushforward $i_\ast$ will be constructed as follows. An \emph{equivariant basic Thom form} is a closed form $\tau \in \Omega^d_{G,c}(U, \CF)$ satisfying $p_\ast \tau = 1$.
We will give a construction of equivariant basic Thom forms at the end of this section. Suppose for now that an equivariant basic Thom form has been constructed. 
Then we define the basic equivariant pushforward as the composition 
\begin{equation} i_*: \Omega^k_G(A, \CF) \stackrel{p^\ast}{\to} \Omega^k_G(U, \CF) \stackrel{\wedge \tau}{\to} \Omega_{G,c}^{k+d}(U, \CF) \to \Omega_G^{k+d}(M, \CF), \label{eq basic equiv pushforward}\end{equation}
where the last arrow denotes extension by zero.

\begin{proposition} The basic equivariant pushforward satisfies, for all closed forms $\beta \in \Omega_G(A,\CF)$ and $\eta \in \Omega_G(U,\CF)$ 
\[ \int_M \eta \wedge i_\ast \beta = \int_A i^\ast \eta \wedge \beta. \]
\end{proposition}
\proof $i_\ast \beta = p^*\beta \wedge \tau$ is a form compactly supported in an invariant neighborhood $U$ of $A$. Therefore we have
\allowdisplaybreaks
\begin{align*} 
\int_M \eta \wedge i_\ast \beta &= \int_U \eta \wedge p^\ast\beta \wedge \tau &\textrm{(by definition of $i_\ast$)}&\\
 &= \int_U p^\ast i^\ast \eta \wedge p^\ast \beta \wedge \tau &  \textrm{(by Lemma \ref{lemma-bundle-cohomology})}&\\
 &= \int_A i^\ast \eta \wedge \beta \wedge p_\ast \tau & \textrm{(by Lemma \ref{lemma-pushforward})}&\\
 &= \int_A i^\ast \eta \wedge \beta & \textrm{(by $p_\ast\tau=1$)}.& \qed
\end{align*}
As in \cite{GS99}, we obtain that the induced map on cohomology $p_*: H^k_{G,c}(U, \CF)$ $\to H^{k-d}_G(A,\CF)$ is an isomorphism with inverse $i_*$.

It remains to construct the equivariant basic Thom form. We use a variant of the Mathai-Quillen construction based on the presentation in \cite[Chapter~10]{GS99} (see also \cite{ToebenLoc, GNTlocalization} for closely related constructions). 
First we identify $U$ with the normal bundle $N \to A$, equipped with a $G \times T$-invariant metric. Let $P \to A$ denote the bundle of oriented orthonormal frames of $N$: it is a $G \times T$-equivariant principal $SO(d)$-bundle over $A$. Consider the map $P \times \R^d \rightarrow N$,
\[ (x, (e_1, \dots, e_d), v) \to (x, v_1 e_1 + \cdots + v_d e_d). \]
It gives a $G \times T$-equivariant diffeomorphism $(P \times \mathbb{R}^d) / SO(d) \cong N$. Equip $P$ with a $G\times T$-invariant basic connection form. Recall that such a form has to exist, see Remark \ref{rem existence basic con form}. 
Using the Cartan model of equivariant basic cohomology, the Cartan map yields isomorphisms 
\allowdisplaybreaks
\begin{align*}
  \phi_N:& \Omega_{SO(d) \times G,c}(P \times \mathbb{R}^d, \mathcal{E}\times \{*\}) \stackrel{\cong}{\to} \Omega_{G,c}(N, \CF)  \\
  \phi_A:&\Omega_{SO(d) \times G}(P, \mathcal{E}) \stackrel{\cong}{\to} \Omega_G(A, \CF),
\end{align*}
where $\mathcal{E}$ denotes the foliation induced by $R$ on $P$.
Let $p_2: P \times \mathbb{R}^d \to \mathbb{R}^d$ be the projection. We define $\tau$ by
\[ \tau := \phi_N(p_2^\ast(\nu \otimes 1)) \in \Omega_{G,c}(N, \CF) \]
where $\nu \in \Omega_{SO(d),c}(\R^d)$ is the (modified) universal Thom-Mathai-Quillen form as constructed in \cite[\S 10.3]{GS99}, $\nu\otimes 1 \in \Omega_{SO(d)\times G,c}(\R^d)$. By analogous arguments to \cite[\S 10.4]{GS99}, we hence have the following.

\begin{theorem} \label{thm-thom}
The form $\tau \in \Omega_{G,c}^d(U, \CF)$ as constructed above is a Thom form for the projection $p: U \to A$. Consequently, the basic equivariant pushforward $i_*: H_G^k(A, \CF) \to H_G^{k+d}(M, \CF)$ is well-defined.
\end{theorem}

\subsection{The localization formula}\label{subsec loc formula}

In this section, we would like to derive a basic version of an Atiyah-Bott-Berline-Vergne type localization formula. We follow the line of proof in \cite{AtiyahBottMomentMap}, adjusting it to the basic setting. 
We assume throughout that the $G$-fixed points have closed Reeb orbits. 
Then $\Crit (\mu)$, the minimal, 1-dimensional $G\times \{\phi_t\}$-orbits, are the 1-dimensional $G\times T$-orbits.
This as\-sump\-tion is obviously satisfied if all Reeb orbits are closed or 
if there are no $G$-fixed points. Note that the later is the case if 0 is a regular value of the contact moment map $\mu$. 
Throughout this section, we work with cohomology with complex coefficients. Then $S(\fg^*)=\C[u_1, ..., u_s]$, where the $u_i$ are coordinates of $\fg^*\otimes \C$. We will make use of the notion of the \emph{support} of a finitely generated module. Recall that in the special case of a module $H$ over $\C[u_1, ..., u_l]$, the support is the subset of $\C^l$ defined by:
\[\Supp H = \bigcap_{\stackrel{f \in \C[u_1, ..., u_l]}{fH=0}} V_f,\] 
where $V_f= \{u \in \C^l \mid f(u) = 0\}$. In particular, a free module has the whole space $\C^l$ as support. An element $h \in H$ is called a \emph{torsion} element if there is a $0\neq f \in \C[u_1, ..., u_l]$ with $fh=0$. If all elements are torsion elements, then $H$ is called a \emph{torsion module}. Note that $H$ is a torsion module if and only if $\Supp H$ is a proper subset of $\C^l$.
For more details, the reader is referred to  \cite[Section 3]{AtiyahBottMomentMap} and the reference therein.

\begin{definition} For $x \in M$ we denote by $\fg_x$ and $\widetilde\fg_x$ the \emph{stabilizer algebra} and \emph{generalized stabilizer algebra}, respectively, to be
\begin{align*}
  \fg_x= \{\xi \in \fg  \ | \  \xi_M(x) = 0 \}, \qquad  \widetilde\fg_x = \{\xi \in \fg  \ | \  \xi_M(x) \in \R R(x) \},
\end{align*}
where $R(x) \in T_x M$ denotes the Reeb vector at $x$.
\end{definition}

Then $\Crit(\mu)=\{x \in M \mid \widetilde \fg_x = \fg\}$. By our assumption, $\Crit(\mu)$ is the union of the 1-dimensional $G\times T$-orbits, and every connected component is a closed submanifold of even codimension (cf. \cite[Lemma~9]{Casselmann2016}). 

\begin{lemma}\label{lemma-support-orbit} 
Let $O = (G \times T) \cdot x$ be an orbit and suppose that $U \subseteq M$ is a $G\times \{\phi_t\}$-invariant submanifold admitting a $G \times \{\psi_t\}$-equivariant map $p: U \to O$. Then \[\Supp H_G(U, \CF) \subseteq \widetilde \fg_x\otimes \C.\]
\end{lemma}
\begin{proof} First note that the $S(\fg^\ast)$-algebra structure on $H_G(U, \CF)$ factors as
\[ S(\fg^\ast) \to H_G(O, \CF) \to H_G(U, \CF), \]
whence $\Supp H_G(U, \CF) \subseteq \Supp H_G(O, \CF)$. Thus, it suffices to show
\[\Supp H_G(O, \CF) \subseteq \widetilde \fg_x\otimes \C.\]
For all $h \in G \times T$, we have $\widetilde \fg_{h \cdot x} = \widetilde \fg_x$. 
In particular, the generalized stabilizer is constant along $O$. Let $\fk$ be a complement of $\widetilde \fg_x$ in $\fg$ such that $\mathfrak{k}$ is the Lie algebra of a subtorus $K$ of $G$. Since $\widetilde \fg_x$ acts trivially on $\Omega (O, \CF)$, the Cartan complex can be written as $\Omega_G(O, \CF) = S(\widetilde \fg_x^*)\otimes S(\mathfrak{k}^*) \otimes \Omega (O, \CF)^{K}$ and $d_G = 1 \otimes d_{K}$, hence $H_G(O, \CF)=S(\widetilde \fg_x^*)\otimes H_{K}(O, \CF)$. 
$K$ acts locally freely and transversally on $O$, so $\Omega(O, \CF)$ is a $\mathfrak{k}$-dga  of type (C) (cf.~\cite[Def. 2.3.4]{GS99}) and $H_{K}(O, \CF)=  H(\Omega(O, \CF)_{ \mathrm{bas } \mathfrak{k}})$ (cf.~\cite[\S~4.6]{GS99} and \cite[Lemma~3.18]{GT2016equivariant}). It also follows that $K\times \{ \phi_t\}$ acts locally freely on $O$ so that the orbits of this action define a foliation $\mathcal{E}$ of $O$. Since $G\times T$ is compact, we can, in particular, find a metric with respect to which the $K\times \{ \phi_t\}$-action is isometric. Hence, $\mathcal{E}$ is a Riemannian foliation (cf. also \cite[p.~100]{Molino}). This, however, means that the basic cohomology $H(O, \mathcal{E})=H(\Omega(O, \CF)_{ \mathrm{bas } \mathfrak{k}})$ is of finite dimension by \cite[Th\'eor\`{e}me~0]{KSH85}. Therefore, the support of $H_G(O, \CF)$ is contained in $\widetilde \fg_x\otimes \C$.
\end{proof}

\begin{proposition}
	Let $X$ be a closed $G\times T$-invariant submanifold of $M$. Then the supports of $H_{G}^*(M \setminus X, \CF)$ and $H_{G,c}^*(M \setminus X, \CF)$ lie in 
	$\cup_{x \in M\setminus X} \widetilde \fg_x\otimes \C.$
	Note that since only finitely many different $\widetilde \fg_x$ occur on $M$, this is a finite union.
\end{proposition}
\begin{proof}
	(cf. \cite[Proposition~3.4]{AtiyahBottMomentMap} and the proof thereof in \cite[Theorem~11.4.1]{GS99})
	Let $U$ be a $G\times T$-invariant tubular neighborhood of $X$. 
	By cohomology equivalence, it suffices to proof the assertion for $H_G(M\setminus \bar U, \CF)$. Since $M\setminus U$ is compact, we can cover $M\setminus \bar U$ with $N$ tubular neighborhoods $U_i$ of $G\times T$-orbits of points $x_i \in M\setminus U \subset M \setminus X$. Let $V_s = U_1 \cup ... \cup U_{s-1}$. 
	Using Lemma \ref{lemma-support-orbit} together with the equivariant basic Mayer-Vietoris sequence for $U_s$ and $V_s$ (cf.~\cite[Proposition~6]{Casselmann2016}; the sequence for the compactly supported case is obtained analogously by adjusting the proof of \cite[Proposition 2.7]{bott2013differential} to the equivariant basic setting), the claim follows by induction.
\end{proof}

Let $C:= \Crit(\mu)$.
The previous result then immediately yields the following:

\begin{corollary} \label{cor-torsion-module}
	The supports of $H_{G}^*(M \setminus C, \CF)$ and $H_{G,c}^*(M \setminus C, \CF)$ lie in $\bigcup_{\widetilde \fg_x \neq \fg} \widetilde \fg_x\otimes \C$. In particular, $H_{G}^*(M \setminus C, \CF)$ is a torsion module over $S(\fg^*)$.
\end{corollary}
	
	The same holds for any $G\times \{\phi_t\}$-invariant 
	subset of $M \setminus C$ and, by exactness, for the relative 
	equivariant basic cohomology of any pair in $M \setminus C$. 
		
\begin{theorem}
	Denote by $i : C \hookrightarrow M$ the inclusion. Then the kernel and cokernel of $i^* : H^*_G(M, \CF) \rightarrow H^*_G(C, \CF)$ have support in $\bigcup_{\widetilde \fg_x \neq \fg} \widetilde \fg_x\otimes \C$. In particular, both $S(\fg^*)$-modules have the same rank, $\dim H^*(C, \CF)$, and $\ker i^*$ is exactly the module of torsion elements in $H_G(M, \CF)$.
\end{theorem}
\begin{proof} From the long exact sequence for the pair $(M, C)$, one sees immediately that $\ker i^\ast$ is a quotient module of $H_G(M, C, \CF)$, and that $\coker i^\ast$ is a sub-module of $H_G(M, C, \CF)$. 
But $H_G(M, C, \CF)$ is a torsion module with support in $\bigcup_{\widetilde \fg_x \neq \fg} \widetilde \fg_x\otimes \C$ by Corollary \ref{cor-torsion-module} and Proposition \ref{prop rel cohom iso comp supp}.
Since $H^*_G(C, \CF)= S^*(\fg^*)\otimes H^*(C, \CF)$ is a free $S^*(\fg^*)$-module, the rank statement follows and every torsion element has to be mapped to zero under $i^*$.
\end{proof}

\begin{proposition}
	The kernel and cokernel of the push forward $i_*\!\! :\! H_G(C, \CF) $ $\to H_G(M, \CF)$ have support in $\bigcup_{\widetilde \fg_x \neq \fg} \widetilde \fg_x\otimes \C$ and are therefore torsion. 
\end{proposition}
\begin{proof}
	Let $U_j$ denote a sufficiently small invariant tubular neighborhood of the connected component $C_j \subset C$ such that $U_j \cap U_i = \varnothing$ for $i\neq j$ and set $U=\cup U_j$. Then $\partial U_j$ is a sphere bundle over $C_j$, in particular, a smooth manifold, and $G\times T$-invariant. 
	Note that Definition \ref{def rel cohom} and Propositions \ref{prop LES cohom} and \ref{prop rel cohom iso comp supp} extend to include closed subsets that are $G\times \{\phi_t\}$-invariant open submanifolds with invariant boundary. We consider the long exact sequence of the pair $(M, M\setminus U)$. By the Thom isomorphism, $H_G(C, \CF) \cong H_{G,c}(U, \CF)$ and $ H_{G,c}^*(U, \CF) \cong H^*_G(M, M \setminus U, \CF)$ by Proposition \ref{prop rel cohom iso comp supp}. 
	The long exact sequence then yields that $\ker \iota_*$ is the image of a torsion module with support in $\bigcup_{\widetilde \fg_x \neq \fg} \widetilde \fg_x\otimes \C$ 
	and that $\coker \iota_*$ is isomorphic to the image of $H^*_G(M, \CF)\to H^*_G(M \setminus U, \CF)$, a submodule of a torsion module with support in $\bigcup_{\widetilde \fg_x \neq \fg} \widetilde \fg_x\otimes \C$.
\end{proof}

From the preceding two statements, it follows that $i^* i_* :H_G(C, \CF)\rightarrow H_G(C, \CF)$ is an isomorphism modulo torsion.
As in \cite[Chapter~10.5]{GS99}, we obtain that $i^* i_* = e_G(\nu C,\CF)$ is the multiplication with the basic $G$-equi\-va\-ri\-ant Euler class of the normal bundle of $C$ (cf.~, e.g., \cite[Definition~11]{Casselmann2016}). Hence, $e_G(\nu C,\CF)$ is invertible in the localized module.

\begin{remark}
	Alternatively, it can be shown directly that $e_G(\nu C,\CF)$ is not a zero divisor in $H^*_G(C, \CF)$, see \cite[Lemma~13]{Casselmann2016}.
\end{remark}

\begin{theorem} \label{thm-duistermaat-heckman} For all $\eta \in H_G(M, \CF)$ we have the exact integration formula
\begin{equation*} \int_M \alpha \wedge \eta = \sum_{C_j \subseteq C} \int_{C_j} \frac{i_j^\ast(\alpha \wedge \eta)}{e_G(\nu C_j,\CF)}, \end{equation*}
where $C_j \subseteq C$ denote the connected components and  $i_j: C_j \hookrightarrow M$ their inclusions.
\end{theorem}
\begin{proof} 
 The inverse of $i_*$ on the localized module is given by
$Q := \sum_{C_j \subseteq C} \frac{i_j^*}{e_G(\nu C_j,\CF)}$.
We therefore obtain for every $\eta \in H_G(M, \CF)$ 
\begin{equation} \label{eqn-q-integration} \int_M \alpha \wedge \eta = \int_M \alpha \wedge i_\ast Q \eta. \end{equation}
Now, using the definition of $i_\ast$ in terms of Thom forms we can express $\eta$ as
\begin{equation} \label{eqn-q-integration2}
\eta = i_\ast Q \eta = \sum_j (i_j)_\ast \frac{i_j^\ast \eta}{e_G(\nu C_j,\CF)} = \sum_j p_j^\ast \left( \frac{i_j^\ast \eta}{e_G(\nu C_j, \CF)} \right) \wedge \tau_j, \end{equation}
where $\tau_j$ is an equivariant basic Thom form compactly supported in a small $G\times T$-invariant tubular neighborhood $U_j$ of $C_j$, $p_j: U_j \to C_j$ is the projection. By Lemma \ref{lemma-bundle-cohomology}, we have
\begin{align*}
\int_{U_j} \alpha \wedge p_j^\ast \left(\frac{i_j^\ast \eta}{e_G(\nu C_j,\CF)} \right) \wedge \tau_j  
 &= \int_{U_j} p_j^\ast \left( \frac{i_j^\ast(\alpha \wedge \eta)}{e_G(\nu C_j,\CF)} \right) \wedge \tau_j \\
 &= \int_{C_j}  \frac{ i_j^\ast(\alpha \wedge \eta)}{e_G(\nu C_j,\CF)} \wedge (p_j)_\ast \tau_j.
\end{align*}
Since $(p_j)_\ast \tau_j = 1$, we obtain the desired integration formula by summing over $j$ and using the identites \eqref{eqn-q-integration}-\eqref{eqn-q-integration2}.
\end{proof}

\section{Equivariant Integration Formulae}

\subsection{Equivariant integration}
Following an idea of Witten \cite{WittenNL}, Jeffrey and Kirwan \cite{JeffreyKirwan95} proved analogues of Theorems \ref{thm-integration} and \ref{thm-residue} for symplectic quotients. 
By far the most important ingredient in their proof is the Atiyah-Bott-Berline-Vergne integration formula \cite{AtiyahBottMomentMap, BerlineVergne}, as this essentially allows the problem to be reduced to studying the properties of Gaussian integrals over the vector space $\fg$. 
Armed with our localization formula (Theorem \ref{thm-localization}) and the local normal form of the moment map (Proposition \ref{prop-normal-form}), we will obtain the $K$-contact analogues, Theorems \ref{thm-integration} and \ref{thm-residue}, by the same line of argumentation as Jeffrey-Kirwan.

 Let $\eta$ be a form representing a class in $H_G(M, \CF)$ and denote by $\Pi_\ast: H_G(M, \CF) \to H_G$ the basic equivariant pushforward 
$\Pi_\ast \eta = \int_M \alpha \wedge \eta$. 
We will apply $\Pi_*$ to classes of type $\eta \wedge e^{id_G\alpha}$, which are not equivariant basic cohomology classes according to our definition, since they are not polynomial but analytic in $\phi$. This is well defined, provided one replaces the codomain with a suitable completion of $H_G$.
With this in mind, for any closed equivariant basic form $\eta$, with $s=\dim \fg$ and $\epsilon >0$, we consider the integral
\begin{equation*} \label{eqn-ieps-defn}
  I^\eta(\epsilon) = \frac{1}{(2\pi i)^s \vol(G)} \int_{\fg} e^{-\epsilon|\phi|^2/2} (\Pi_\ast (\eta  \wedge  e^{id_G\alpha}))(\phi) d\phi,
\end{equation*}
where $d\phi$ is a measure on $\fg$ corresponding to a metric on $\fg$ that induces a volume form $\vol_G$ on $G$, $\vol(G)=\int_G \vol_G$. Then $d\phi/\vol (G)$ is independent of that choice.
Note that $I^\eta(\epsilon)$ is well defined; $\eta \wedge e^{id_G\alpha}$ is only of mild exponential dependence on $\phi$ so that the factor $e^{-\epsilon|\phi|^2/2}$ ensures convergence of the integral.

Following Jeffrey-Kirwan, we will relate the $\epsilon \to 0$ asymptotics of $I^\eta(\epsilon)$ to intersection pairings on the contact quotient. First, we must rewrite $I^\eta(\epsilon)$ in a more convenient form. 
For any (tempered) distribution on $\fg$, introduce the Fourier transform
\begin{equation*}
  (\BF f)(z) = (2\pi)^{-s/2} \int_\fg f(\phi) e^{-iz(\phi)} d\phi.
\end{equation*}
By definition, $\BF(f)$ is naturally a distribution on $\fg^\ast$.  Set 
\begin{equation*}
  Q^\eta(y) = \BF\left[ \Pi_\ast (\eta \wedge  e^{id_G\alpha}) \right](y).
\end{equation*}
Let $g_\epsilon$ denote the Gaussian function $g_\epsilon(\phi) = e^{-\epsilon |\phi|^2 / 2}$, with Fourier transform $(\BF g_\epsilon)(z) = \epsilon^{-s/2} g_{\epsilon^{-1}}(z).$
Note that $I^\eta(\epsilon)$ can be viewed as the $L^2$ inner product of the functions $g_\epsilon(\phi)$ and $\Pi_\ast(\eta \wedge  e^{id_G\alpha})(\phi)$. Since the Fourier transform is an $L^2$ isometry, we have the following identity.

\begin{lemma} \label{lemma-ie-fourier}
\[ I^\eta(\epsilon) = \frac{1}{(2\pi i)^s \epsilon^{s/2} \vol(G)} \int_{\fg^*} Q^\eta(y) e^{-|y|^2 /2\epsilon} dy. \]
\end{lemma}

\begin{lemma}\label{lemma Q compact supp}
	The distribution $Q^\eta(y)=\BF\left[ \Pi_\ast (\eta \wedge  e^{id_G\alpha}) \right](y)$ can be expressed as follows 
	\begin{align*}
		Q^\eta(y) =(2\pi)^{s/2} \sum_J i^J \frac{\partial}{\partial y^J} \int_M \alpha \wedge \eta_J \wedge  e^{id\alpha} \delta(-\mu-y),
	\end{align*}
	where $\eta= \sum_J \eta_J y^J$, summing over multi indices $J$, with $y^j$ denoting an or\-tho\-nor\-mal basis of $\fg^*$ and $\eta_J \in \Omega^*(M, \CF)$, and $\delta$ denotes the Dirac delta distribution.
	In particular, $Q^\eta(y)$ is supported in the compact set $-\mu(M)$.
\end{lemma}
\begin{proof}
We make use of the arguments given in \cite[\S~5,~7]{JeffreyKirwan95}. Write $\eta(\phi) = \sum_J \eta_J \phi^J$ with $\phi^j$ denoting the coordinate functions $y^j(\phi)$. Recalling the definition of $Q^\eta(y)$, we have
\allowdisplaybreaks
\begin{align*}
  Q^\eta(y) &=\BF\left[ \Pi_\ast (\eta \wedge  e^{id_G\alpha}) \right](y)\\
  &= \frac{1}{(2\pi)^{s/2}}\sum_J \int_M \int_\fg \alpha \wedge \eta_J \phi^J \wedge  e^{id\alpha - i\ev{\mu, \phi} - i \ev{y,\phi}} d\phi \\
  &= \frac{1}{(2\pi)^{s/2}} \sum_J i^J \frac{\partial}{\partial y^J} \int_M \alpha \wedge \eta_J \wedge  e^{id\alpha} \int_\fg  e^{i\ev{-\mu-y, \phi}} d\phi \\
  &= (2\pi)^{s/2} \sum_J i^J \frac{\partial}{\partial y^J} \int_M \alpha \wedge \eta_J \wedge  e^{id\alpha} \delta(-\mu-y).
\end{align*}
This shows that $Q^\eta(y)$ is the integral over $x \in M$ of a distribution $S(x,y)$ on $M \times \fg$ which is supported on the set $\{(x,y)  \ | \  - \mu(x)=y \}$. 
\end{proof}

\begin{proposition} \label{prop-fourier-polynomial}
The distribution $Q^\eta(y)$ may be represented by a piecewise polynomial function.
\end{proposition}
\begin{proof}
 
Let $C_j$ denote a connected component of the critical set $C=\operatorname{Crit} \mu$ of codimension $d$. 
By Proposition \ref{prop crit mu}, $G\times T$ acts in $R$-direction only and its isotropy $(\fg\times \ft)_{C_j}$ has codimension 1. Let $\theta$ be a $G\times T$-invariant, basic connection form on the bundle of oriented orthonormal frames of $\nu C_j$ and denote by $F^\theta$ its (ordinary) curvature. 
Choose a basis $(X_i)$ of $\fg\times \ft$ such that $X_1, ..., X_{r-1}$ is a basis of $(\fg\times \ft)_{C_j}$ and $X_r=R$. Denote its dual basis by $u_i$. Then, since $\theta$ is basic, $\iota_{X_r}\theta=0$. The basic $G\times T$-equivariant Euler form is then given by
\begin{align}
	e_{G\times T}(\nu C_j,\CF) &= \operatorname{Pf}\left(F^\theta - \sum_i \iota_{X_i}\theta u_i\right) =\operatorname{Pf}\left(F^\theta - \sum_{i=1}^{r-1} \iota_{X_i}\theta u_i\right).\label{eq e GxT}
\end{align}
 Denote by $(G \times T)_{C_j}\subset G\times T$ the subtorus that has $(\fg\times \ft)_{C_j}$ as Lie algebra. $\nu C_j$ is a $(G \times T)_{C_j}$-equivariant vector bundle over $C_j$. 
 By the splitting principle for equivariant bundles, we may assume that the normal bundle splits as a direct sum of line bundles $\nu C_j = \oplus_i L_i$ and $(G \times T)_{C_j}$ acts on $L_i$ with weight $\beta^j_i$. 
 Then the basic $(G \times T)_{C_j}$-equivariant Euler form factors as $e_{(G \times T)_{C_j}}(\nu C_j,\CF) = \prod_i e_{(G \times T)_{C_j}}(L_i,\CF)$ and $2\pi e_{(G \times T)_{C_j}}(L_i,\CF) = c^j_i + \beta^j_i$, where $c^j_i \in \Omega^2(C_j,\CF)$ is the (ordinary) basic Euler form of $L_i$. Hence,
  \begin{align}
 	(2\pi)^{d/2}e_{(G \times T)_{C_j}}(\nu C_j,\CF) &=\prod_{i=1}^{d/2} (c^j_i + \beta^j_i). \label{eq e GxT Cj}
 \end{align}
 But we can also compute $e_{(G \times T)_{C_j}}(\nu C_j,\CF)$ as $\operatorname{Pf}\left(F^\theta - \sum_{i=1}^{r-1} \iota_{Y_i}\theta b_i\right)$, where $(Y_i)$ denotes a basis of $(\fg\times \ft)_{C_j}$ and $(b_i)$ its dual basis.
\eqref{eq e GxT} yields that if we extend $e_{(G \times T)_{C_j}}(\nu C_j,\CF)$ to all of $\fg \times \ft$ by setting it equal to 0 on $\R R$, we obtain $e_{G \times T}(\nu C_j,\CF)$. 
Hence, extending $\beta^j_i \in (\fg\times \ft)_{C_j}^*$ 
and combining \eqref{eq e GxT} and \eqref{eq e GxT Cj} yields
\begin{align}
	(2\pi)^{d/2}e_{G\times T}(\nu C_j,\CF) = \prod_{i=1}^{d/2}(c^j_i + \beta^j_i).\label{eq e GxT is prod}
\end{align}
The definition of the Euler form yields that $e_{G}(\nu C_j,\CF)$ is exactly given by the restriction of $e_{G\times T}(\nu C_j,\CF)$ to $\fg$ so that, by \eqref{eq e GxT is prod},
\begin{align}
	(2\pi)^{d/2}e_G(\nu C_j,\CF) = \prod_{i=1}^{d/2}(c^j_i + \beta^j_i|_{\fg}).\label{eq e G is prod}
\end{align}
By Theorem \ref{thm-localization}, we have that 
\begin{equation}
  \Pi_\ast ( \eta \wedge  e^{id_G\alpha}) = \sum_{j} \int_{C_j} \frac{i_j^\ast(\alpha \wedge \eta  \wedge e^{id_G\alpha})}{e_G(\nu C_j,\CF)}.\label{eq PI eta e idG alpha as sum}
\end{equation}
It now follows with \eqref{eq e G is prod}, by the same argument as \cite[Lemma 2.2]{JeffreyKirwan95}, that the pushforward may be written as a sum
\begin{align} \Pi_\ast(\eta \wedge  e^{id_G\alpha})(\phi) = \sum_{j} \sum_{a \in \CA_j}\frac{e^{-i\mu(C_j)(\phi)}\int_{C_j} i_j^*(\alpha \wedge \eta(\phi) \wedge  e^{id\alpha})c_{j,a}}{\prod_{i} (\beta_{i}^j|_\fg)(\phi)^{n_{j,i}(a)}},\label{eq pushforward sum}
\end{align}
where $\CA_j$ is a finite indexing set, $c_{j,a}\in H^*(C_j,\CF)$ is determined by the $c^j_i$, and $n_{j,i}(a)$ is a non-negative integer.
In particular, for every $(j,a)$, the term on the right hand side of Equation \eqref{eq pushforward sum} is given by the product of $\frac{e^{-i\mu(C_j)(\phi)}}{(\prod_i \beta_i^j|_\fg)(\phi)^{n_{j,i}(a)}}$ with a polynomial in $\phi$, where the polynomial is simply a constant if $\eta = 1$.
Given this description of the pushforward, the piecewise polynomial property of $Q^\eta(y)$ for $\eta=1$ follows from the same argument as \cite[Theorem~4.2]{JeffreyKirwan95}, making use of Lemma \ref{lemma Q compact supp}. 
For arbitrary $\eta$, it follows from the case $\eta=1$, noting that every $(j,a)$-summand contributes a piecewise polynomial function, applying that - up to a factor of $(-i)$ - Fourier transformation interchanges differentiation and multiplication by a coordinate, cf.~\cite[Lemma~7.1.3]{hoermander}.
\end{proof}

\subsection{Asymptotic analysis}

By Lemma \ref{lemma-ie-fourier} and Proposition \ref{prop-fourier-polynomial}, we are reduced to estimating the asymptotics of an integral of the form
$I(\epsilon)=\int e^{-|y|^2 / 2\epsilon} Q(y) dy$,
where $Q(y)$ is piecewise polynomial. Suppose that $Q(y)$ is regular near the origin, and let $Q_0(y)$ denote the polynomial which agrees with $Q(y)$ near the origin. Set
\begin{equation*}
  I_0(\epsilon) = \frac{1}{(2\pi i)^s \epsilon^{s/2} \vol (G)} \int_{\fg^\ast} Q_0(y) e^{-|y|^2 /2\epsilon} dy.  
\end{equation*}

\begin{lemma} \label{lemma-polynomial-asymptotics} Suppose that $Q(y)$ is regular near the origin and define $I(\epsilon)$ and $I_0(\epsilon)$ as above. Then we have the asymptotic
\[ |I(\epsilon) - I_0(\epsilon)| = o(\epsilon^{-s/2} e^{-c/\epsilon}) \]
for some constant $c > 0$.
\end{lemma}
\begin{proof} 
Let $R(y) = Q(y)-Q_0(y)$. Then $R(y)$ is piecewise polynomial and identically zero in a neighborhood of the origin. Pick $\delta > 0$ so that $R(y)$ is identically zero for $|y|<\delta$. Switching to polar coordinates, we have
\[ |I(\epsilon) - I_0(\epsilon)| \leq c' \epsilon^{-s/2}\int_{S^{s-1}}\int_\delta^\infty |R(y)| e^{-r^2 / 2\epsilon} r^{s-1} dr dv_{S^{s-1}}, \]
where $c'$ is a constant that does not depend on $\epsilon$.
Since $R(y)$ is piecewise polynomial, we can find constants $a_0, \dots, a_N$ so that for $|y| > \delta$, we have $|R(y)| \leq \sum_j a_j |y|^j$. Combining this with the previous estimate, we have
\[ |I(\epsilon) - I_0(\epsilon)| \leq c^{''} \epsilon^{-s/2} \sum_{j=1}^N a_j \int_\delta^\infty r^{j+s-1} e^{-r^2 / 2\epsilon} dr, \]
where $c^{''}$ is a constant that does not depend on $\epsilon$. This reduces the problem to estimating integrals of the form $\int_\delta^\infty r^\ell e^{-r^2 / 2\epsilon} dr$ for $\ell \geq 0$. The following Lemma \ref{lemma int bound} shows that such an integral is bounded by a function of the form 
$p(\sqrt{2\epsilon}) e^{-\delta^2/(4\epsilon)}$, where $p$ is a polynomial of degree $\ell +1$. 
The result follows.
\end{proof}

\begin{lemma}\label{lemma int bound}
	The integral $I_n^\delta (a):= \int_\delta^\infty x^n e^{-ax^2} dx$, $a,\delta >0$, $n \in \mathbb{N}$, is bound\-ed from above by a function of the form $p_n(1/\sqrt{a}) e^{-\frac{\delta^2 a}{2}}$, where $p_n$ is a polynomial of degree $n+1$.
\end{lemma}
\begin{proof}
	The claim is shown by induction on $n$. By substituting $x = \sqrt{a}^{-1}y$, we obtain
	\begin{align*}
		(I_0^\delta(a))^2 &= \left(\sqrt{a}^{-1} \int \limits_{\sqrt{a}\delta}^\infty e^{-y^2} dy \right)^2= \left(\frac{1}{2\sqrt{a}} \int\limits_{\R \setminus [-\sqrt{a}\delta,\sqrt{a}\delta]} e^{-y^2} dy \right)^2 \\
		&= \frac{1}{4a}  \int\limits_{\R^2 \setminus [-\sqrt{a}\delta,\sqrt{a}\delta]^2} e^{-(x^2+y^2)} dx dy  \leq \frac{1}{4a}   \int\limits_{\R^2 \setminus B_{\sqrt{a} \delta}(0)} e^{-(x^2+y^2)} dx dy  ,
	\end{align*}
	where $B_{\sqrt{a} \delta}(0)$ denotes the ball of radius $\sqrt{a} \delta$, centered at the origin. By passing to polar coordinates, the integral becomes
	\begin{align*}
		(I_0^\delta(a))^2 &\leq
		 \frac{1}{4a}  \int\limits_{0}^{2\pi} \int\limits_{\delta \sqrt{a}}^{\infty} \frac{d}{dr}\left[-\frac{1}{2}e^{-r^2}\right] dr d\phi 
		 =\frac{1}{4a}  \int\limits_{0}^{2\pi} \frac{1}{2}e^{-\delta^2a} d\phi  = \frac{\pi}{4a}(e^{-\frac{\delta^2a}{2}})^2.
	\end{align*}
	For $n=1$, we can directly compute
	\begin{align}
		I_1^\delta(a) &=\int_\delta^\infty x e^{-ax^2} dx = -\frac{1}{2a}\int_\delta^\infty \frac{d}{dx}\left[ e^{-ax^2}\right] dx = \frac{1}{2a}e^{-a\delta^2}\label{eq I1} \leq  \frac{1}{2a}e^{-\frac{a\delta^2}{2}}. \notag
	\end{align}
	Thus, the claim holds for $n=0,1$. Now, let $n \geq 2$ and suppose the claim holds for $n-2$. We integrate by parts.
	\allowdisplaybreaks
	\begin{align*}
		I_n^\delta (a)&=
		\int_\delta^\infty -\frac{ x^{n-1}}{2a}\cdot \frac{d}{dx}\left[ e^{-ax^2}\right] dx \\
		&= \left[-\frac{ x^{n-1}}{2a}\cdot  e^{-ax^2}\right]_{x=\delta}^\infty + \int_\delta^\infty \frac{(n-1) x^{n-2}}{2a}\cdot e^{-ax^2} dx\\
		&= \frac{\delta^{n-1}}{2a}e^{-a\delta^2} + \frac{n-1}{2a}I_{n-2}^\delta (a) \leq  \left(\frac{\delta^{n-1}}{2a} + \frac{n-1}{2a} p_{n-2}(a^{-1/2})\right)e^{-\frac{a\delta^2}{2}}.
	\end{align*}
	Setting $p_n(a^{-1/2})= \left(\frac{\delta^{n-1}}{2a} + \frac{n-1}{2a} p_{n-2}(a^{-1/2})\right)$ yields the claim.
\end{proof}

We now want to apply Lemma \ref{lemma-polynomial-asymptotics} to $Q^\eta$.
It remains to show that $Q^\eta(y)$ is regular near $0$, and to compute the polynomial $Q^\eta_0(y)$ which agrees with $Q^\eta(y)$ near the origin. We will make use of the local normal form we found in \S \ref{sec K contact}. Analogous statements in the symplectic setting can be found in \cite[\S\S~5,~7,~8]{JeffreyKirwan95}.

\begin{proposition} \label{prop-q0} Suppose that $0$ is a regular value of $\mu$. Then $Q^\eta(y)$ is regular in some neighborhood of $0$, and on this neighborhood it coincides with the polynomial $Q^\eta_0(y)$ given by
  \[ Q^\eta_0(y) = i^s (2\pi)^{s/2} \int_{\mu^{-1}(0)} q^\ast(\alpha_0 \wedge \eta_0)\wedge e^{i q^*d\alpha_0 - i y(F_\theta)} \Omega, \]
where $\theta$ is a $G$-invariant basic connection form on the $G$-bundle $q:\mu^{-1}(0)\to \mu^{-1}(0) /G$, $F_\theta$ denotes its curvature form,  $\Omega=\theta_1 \wedge ... \wedge \theta_s$ is the volume form on the $G$-orbits defined by $\theta$, $\eta_0 \in H(M_0, \CF_0)$ represents $i_0^* \eta \in H_G(\mu^{-1}(0), \CF)$, where the inclusion $\mu^{-1}(0) \hookrightarrow M$ is denoted by $i_0$, and $\alpha_0$ denotes the induced contact form on $\mu^{-1}(0)/G=M_0$. Here, $\mu^{-1}(0)$ is en\-dowed with the orientation induced by the volume form $q^*(\alpha_0 \wedge (d\alpha_0)^{n-s}) \wedge \Omega$. In particular, with $n_0$ denoting the order of the regular isotropy of the action of $G$ on $\mu^{-1}(0)$, we have 
\[\int_{M_0} \alpha_0\wedge \eta_0 \wedge  e^{id\alpha_0}=\frac{n_0}{i^s(2\pi)^{s/2}\vol G } \BF \left( \Pi_*(\eta \wedge  e^{id_G\alpha})\right)(0).\]
\end{proposition}

\proof
Recall that 
\allowdisplaybreaks
\begin{align*}
  Q^\eta(y) &=\BF\left[ \Pi_\ast (\eta \wedge  e^{id_G\alpha}) \right](y)= \frac{1}{(2\pi)^{s/2}} \int_M \int_\fg \alpha \wedge \eta( \phi) \wedge e^{id_G\alpha - i \ev{y,\phi}} d\phi 
\end{align*}
By Lemma \ref{lemma Q compact supp}, when $y$ is sufficiently small, we may replace the integral over $M$ by an integral over $U \subset M$, where $U$ is a neighborhood of $\mu^{-1}(0)$. 
Using the normal form of Proposition \ref{prop-normal-form}, we see that for small $y$
\begin{align*}
  Q^\eta(y) &= \frac{1}{(2\pi)^{s/2}}  \int_\fg \int_{\mu^{-1}(0) \times B_h} \alpha \wedge \eta(\phi) \wedge e^{id_G\alpha  - i \ev{y,\phi}} d\phi ,
\end{align*}
where $\mu^{-1}(0) \times B_h$ is canonically oriented by the contact volume form.
Consider the projection $\pi: \mu^{-1}(0) \times B_h \rightarrow \mu^{-1}(0) \times \{0\}$ and the inclusion $i:\mu^{-1}(0) \times \{0\}\rightarrow \mu^{-1}(0) \times B_h $. 
Then $i \circ \pi : \mu^{-1}(0) \times B_h \rightarrow \mu^{-1}(0) \times B_h $ is $G $ $\times T$-e\-qui\-va\-ri\-antly homotopic to the identity and, hence, $i$ induces an isomorphism 
\begin{align*}
	H_G(\mu^{-1}(0) \times B_h, \CF \times \{\text{pt.}\}) &\cong  H_G(\mu^{-1}(0) \times \{0\}, \CF \times \{0\})\\
	&=H_G(\mu^{-1}(0) , \CF).
\end{align*}

Since $[q^* \eta_0]=[i_0^*\eta]$ by definition of $\eta_0$, it is $[\pi^*q^*\eta_0]$ $=$ $[\eta|_{\mu^{-1}(0) \times B_h}]$. 
Therefore, there is a $\gamma \in \Omega_G(\mu^{-1}(0) \times B_h, \CF \times \{\text{pt.}\})$ such that $\eta -\pi^* q^* \eta_0 = d_G\gamma$. Set 
\begin{align*}
\Delta :&= Q^\eta(y) - \frac{1}{(2\pi)^{s/2}}  \int_\fg \int_{\mu^{-1}(0) \times B_h} \alpha \wedge \pi^*q^*\eta_0 \wedge e^{id_G\alpha  - i \ev{y,\phi}} d\phi\\
&= \frac{1}{(2\pi)^{s/2}}  \int_\fg \int_{\mu^{-1}(0) \times B_h} \alpha \wedge d_G\gamma \wedge  e^{id_G\alpha  - i \ev{y,\phi}} d\phi.
\end{align*}
Since $d_Gd_G\alpha = 0$ and $d_G\phi_j = 0$, we have 
$d_G\gamma \wedge e^{id_G\alpha - i \ev{y,\phi}}=d_G\left(\gamma \wedge e^{id_G\alpha - i \ev{y,\phi}}\right)$. 
The integral over $\mu^{-1}(0) \times B_h$ picks up only those components of the basic form $d_G\left(\gamma \wedge e^{id_G\alpha - i \ev{y,\phi}}\right)$ of degree $2n$, so we can pass to the ordinary differential and 
\begin{align*}
(2\pi)^{s/2}\Delta &=  \int\limits_\fg \int\limits_{\mu^{-1}(0) \times B_h} \alpha \wedge d(\gamma \wedge e^{id_G\alpha  - i \ev{y,\phi}}) d\phi\\
& = \int\limits_\fg \int\limits_{\mu^{-1}(0) \times B_h} \!-d\left(\alpha \wedge \gamma \wedge  e^{id_G\alpha -i \ev{y,\phi}}\right) + d\alpha \wedge \gamma \wedge e^{id_G\alpha  - i \ev{y,\phi}} d\phi.
\end{align*}
The second summand is basic, hence, its top degree part is zero. Thus, the whole summand vanishes under integration. By Stokes' Theorem, denoting the boundary of $B_h$ by $S_h$, we obtain
\begin{align*}
(2\pi)^{s/2}\Delta &= -  \int_\fg \int_{\mu^{-1}(0) \times S_h} \alpha \wedge \gamma \wedge e^{id_G\alpha  - i \ev{y,\phi}} d\phi.
\end{align*}
Write $\gamma(\phi) = \sum_J \gamma_J \phi^J$. As in the proof of Lemma \ref{lemma Q compact supp}, the previous equation becomes
\begin{align*}
(2\pi)^{s/2}\Delta &= -(2\pi)^{s}\sum_J i^J\frac{\partial}{\partial y^J}\int_{\mu^{-1}(0) \times S_h} \alpha \wedge \gamma_J \wedge e^{id\alpha} \delta(-\mu - y).
\end{align*}
Recall that the local normal form of the moment map is given by $\mu(p,z)=z$. Then, for sufficiently small $y$, $\delta(-\mu-y)$ is supported away from $S_h$ and it follows that $\Delta=0$. This means that, for sufficiently small $y$,
\begin{align*}
 Q^\eta(y)&= \frac{1}{(2\pi)^{s/2}}  \int\limits_\fg \int\limits_{\mu^{-1}(0) \times B_h} \alpha \wedge \pi^*q^*\eta_0 \wedge  e^{id\alpha + i\ev{-\mu - y, \phi}} d\phi\\
 &= (2\pi)^{s/2} \!\!\int\limits_{\mu^{-1}(0) \times B_h} \!\!\alpha \wedge \pi^*q^*\eta_0 \wedge  e^{id\alpha} \delta(-\mu - y)\\
 &=(2\pi)^{s/2} \!\! \int\limits_{\mu^{-1}(0) \times B_h} \!\! (q^*\alpha_0 + z(\theta)) \wedge q^*\eta_0 \wedge e^{idq^*\alpha_0 + idz(\theta) + iz(d\theta)} \delta(-z - y).
 \end{align*}
 Let $j$ index an orthonormal basis of $\fg$ resp. $\fg^*$ and set $\Omega = \theta_1 \wedge ... \wedge \theta_s$ the volume form on the $G$-orbits, $[dz]=dz_1 \wedge ... \wedge dz_s$. 
 We only obtain a non-zero contribution from $e^{idz(\theta)}$ from the term containing $(idz(\theta))^s = s! i^s(-1)^{s(s+1)/2}\Omega \wedge [dz]$
 since all the factors $dz_j$ must appear. Additional factors of $\theta$ will wedge to 0 with $\Omega$, so $z(\theta)$ does not contribute to the integral. We obtain 
 
 \begin{align}
 Q^\eta(y)&=i^s(-1)^{\tfrac{s(s+1)}{2}}(2\pi)^{s/2} \hspace{-12pt} \int\limits_{\mu^{-1}(0) \times B_h}\hspace{-12pt} q^*(\alpha_0 \wedge \eta_0) \wedge  e^{iq^*d\alpha_0 + iz(d\theta)} \Omega \delta(-z - y) [dz].\label{eq Qeta0 calculation}
 \end{align}
 
 The orientation on $\mu^{-1}(0) \times B_h$ is canonically given by the contact volume form 
 \begin{align*}
 	(q^*\alpha_0 +z(\theta))  \! \wedge \!  (q^*d\alpha_0 + d(z(\theta) \! ))^{n}
 	=(-1)^{\tfrac{s(s-1)}{2}} \!  \tfrac{n!}{(n-s)!}q^*\alpha_0 \wedge \!  (q^*d\alpha_0 + z(d\theta))^{n-s} \!  \! \wedge  \! \Omega  \! \wedge [dz].
 \end{align*}
 For $z=0$, this volume form differs by a factor of $(-1)^{s(s-1)/2} \! \tfrac{n!}{(n-s)!}$ from the volume form $\nu := q^*(\alpha_0 \wedge d\alpha_0^{n-s})\wedge \Omega\wedge [dz]$. 
 Hence, when changing the orientation of $\mu^{-1}(0) \times B_h$ to that induced by $\nu$ in Equation \ref{eq Qeta0 calculation}, denoting the thusly oriented manifold by $(\mu^{-1}(0) \times B_h)^\nu$, we obtain a factor $(-1)^{s(s-1)/2}$ and obtain
 \begin{align*}
 Q^\eta(y)&=i^s(2\pi)^{s/2} \int_{(\mu^{-1}(0) \times B_h)^\nu} q^*(\alpha_0 \wedge \eta_0) \wedge  e^{iq^*d\alpha_0 + iz(d\theta)} \Omega \delta(-z - y) [dz].
 \end{align*}
 On $B_h$, we consider the orientation induced by $[dz]$ and we endow $\mu^{-1}(0)$ with the orientation induced by $q^*(\alpha_0 \wedge d\alpha_0^{n-s})\wedge \Omega$ so that their product gives the orientation of $(\mu^{-1}(0) \times B_h)^\nu$. We continue our computation by integrating over $B_h$ and obtain
 \begin{align*}
 Q^\eta(y)&=i^s(2\pi)^{s/2} \int_{\mu^{-1}(0) } q^*(\alpha_0 \wedge \eta_0)\wedge  e^{iq^*d\alpha_0 - iy(d\theta)} \Omega\\
 &= i^s(2\pi)^{s/2} \int_{\mu^{-1}(0) } q^*(\alpha_0 \wedge \eta_0) \wedge e^{iq^*d\alpha_0 - iy(F^\theta)} \Omega,
 \end{align*}
 where we have replaced the term $d\theta$ by the curvature form $F_\theta = d\theta + \frac{1}{2}[\theta,\theta]$, which, as above, does not change the value of the integral. Therefore we obtain the claimed expression for $Q^\eta_0(y)$. This is obviously a polynomial in $y$, since only finitely many terms in the power series expansion of $e^{-iy(F_\theta)}$ are non-zero. 
 
 $\mu^{-1}(0)/ G$ is canonically oriented by $\alpha_0 \wedge d\alpha_0^{n-s}$. Hence, together with above orientation on $\mu^{-1}(0)$, the projection $q$ induces the same orientation on the fibers as $\Omega$.  	
 $\Omega$ integrates to $\vol (G)/n_0$ over the fiber, so, when $y=0$, the previous equation becomes 
	\[\BF(\Pi_*(\eta \wedge  e^{id_G\alpha}))(0) =i^s(2\pi)^{s/2}\vol ( G)/n_0 \int_{ \mu^{-1}(0)/G}\alpha_0 \wedge \eta_0 \wedge e^{id\alpha_0}.\qed \]

\begin{proposition} \label{prop-i0-expression}
  Let $\Theta \in H^4(M_0, \CF_0)$ be the class corresponding to the class $-\frac{<\phi,\phi>}{2}\in H_G^4(\mu^{-1}(0),\CF)\simeq H^4(M_0, \CF_0)$ under the Cartan map. Then
  \[ I^\eta_0(\epsilon) = \tfrac{1}{n_0} \int_{M_0} \alpha_0 \wedge \eta_0 \wedge e^{\epsilon \Theta + id\alpha_0}. \] 
\end{proposition}
\begin{proof}
The Cartan map yields $-|F_\theta|^2/2 = q^* \Theta$ in cohomology. By Proposition \ref{prop-q0},
\allowdisplaybreaks
\begin{align*}
  I^\eta_0(\epsilon) 
 &= \frac{1}{(2\pi \epsilon)^{s/2} \vol (G)} \int\limits_{\mu^{-1}(0) \times \fg^\ast} q^\ast(\alpha_0 \wedge \eta_0) \wedge e^{i q^*d\alpha_0 - i y(F_\theta)-|y|^2/2\epsilon} \wedge \Omega \ dy \\
 &=\frac{1}{(2\pi \epsilon)^{s/2} \vol (G)} \int\limits_{\mu^{-1}(0)} q^\ast(\alpha_0 \wedge \eta_0) \wedge e^{i q^*d\alpha_0} \wedge\Omega \int\limits_{\fg^*} e^{ -i y(F_\theta)-|y|^2/2\epsilon}  \ dy \\
 &= \frac{1}{\vol (G)} \int\limits_{\mu^{-1}(0)} q^\ast(\alpha_0 \wedge \eta_0) \wedge  e^{i q^*d\alpha_0 -\epsilon|F_\theta|^2 / 2} \wedge \Omega \quad \text{\small by Gaussian integration} \\
 &=  \tfrac{1}{n_0}\int\limits_{M_0} \alpha_0 \wedge \eta_0 \wedge e^{i d\alpha_0 + \epsilon \Theta}
\end{align*}
since $\Omega $ integrates to $\vol (G)/n_0$ over the fiber.
\end{proof}

Combining Lemma \ref{lemma-polynomial-asymptotics} with Proposition \ref{prop-i0-expression}, we obtain Theorem \ref{thm-integration}.

\subsection{Jeffrey-Kirwan residues}
We briefly recall the Jeffrey-Kirwan re\-si\-due operation. Let $\Lambda \subset \fg$ be a non-empty open cone and suppose that
$\beta_1, \ldots, \beta_N \in \fg^\ast$ all lie in the dual cone $\Lambda^\ast$.
Suppose that $\lambda \in \fg^\ast$ does not lie in any cone of dimension at most
$s-1$ spanned by a subset of $\{\beta_1, \ldots, \beta_N\}$. Let $\{\phi_1, \ldots, \phi_s\}$
be any system of coordinates on $\fg$ and let $d\phi = d\phi_1 \wedge \cdots \wedge d\phi_s$ be 
the associated volume form. Then there exists a residue operation $\jkres^\Lambda$
defined on meromorphic differential forms of the form
\begin{equation}
  h(\phi) = \frac{q(\phi) e^{i\lambda(\phi)}}{\prod_{j=1}^N \beta_j(\phi)} d\phi,
\end{equation}
where $q(\phi)$ is a polynomial. The operation $\jkres^\Lambda$ is linear in its
argument and is characterized uniquely by certain axioms, cf. \cite[Proposition 3.2]{JeffreyKirwan96}.

Theorem \ref{thm-residue} is now a consequence of our localization formula Theorem \ref{thm-localization} and Proposition \ref{prop-q0}.

\begin{proof}[Proof of Theorem \ref{thm-residue}]
$\BF(\Pi_\ast( \eta \wedge  e^{id_G\alpha}))$ is compactly supported by Lemma \ref{lemma Q compact supp}. Hence, \cite[Proposition~8.6]{JeffreyKirwan95} yields that the residue $\jkres^\Lambda(\Pi_\ast (\eta \wedge  e^{id_G\alpha})d\phi)$ is independent of the cone $\Lambda$.  Since $\BF(\Pi_\ast( \eta \wedge  e^{id_G\alpha}))$ is smooth near $0$ by Proposition \ref{prop-q0} and compactly supported, \cite[Pro\-po\-si\-tion~8.7]{JeffreyKirwan95} gives 
 \[i^{-s}(2\pi)^{-s/2}\BF \left( \Pi_*(\eta \wedge  e^{id_G\alpha})\right)(0) = \jkres^\Lambda(\Pi_\ast (\eta \wedge  e^{id_G\alpha})d\phi).\]
 By Pro\-po\-si\-tion \ref{prop-q0}, we then obtain  
  \begin{align*}
    \int_{M_0} \alpha_0 \wedge \eta_0 \wedge  e^{id\alpha_0} 
    &= \frac{n_0}{ \vol G}\jkres^\Lambda(\Pi_\ast (\eta \wedge  e^{id_G\alpha})d\phi).
  \end{align*}
  Using the expression for $\Pi_\ast (\eta \wedge  e^{id_G\alpha})$ provided by Theorem 
\ref{thm-localization}, namely, Equation \eqref{eq PI eta e idG alpha as sum}, we obtain the claimed formula.
\end{proof}

\section{Examples}

\subsection{Boothby-Wang fibrations} \label{sec-boothby}
We now explain how, for certain symplectic manifolds, the known results  may be recovered from our main theorems.

\begin{theorem}[Boothby-Wang \cite{BoothbyWang}]  Suppose that $(N, \omega)$ is a symplectic manifold with integral symplectic form. Then the connection 1-form $\alpha$ on the prequantum circle bundle $M \to N$ is a contact form. 
Conversely, if $(M, \alpha)$ is a compact contact manifold with Reeb vector field that induces an $S^1$-action, then there is an integral symplectic manifold $(N, \omega)$ such that $M$ is the prequantum circle bundle of $N$, with connection 1-form given by $\alpha$.
\end{theorem}

We call such a principal $S^1$-bundle $M \to N$ with connection form $\alpha$ a \emph{Boothy-Wang fibration}. Denote the period of the flow $\phi_t$ of the Reeb vector field $R$ by $2\pi/\tau$. 
We can identify a Reeb orbit $\{\phi_t(x)\}$ with $S^1$ via $e^{it\tau}\mapsto \phi_t(x)$. The transformation formula then yields that the integral of $\alpha$ over an arbitrary Reeb orbit $\{\phi_t(x)\}$ is equal to $2\pi/\tau$. 

\begin{proposition} \label{prop-boothby-identification}
  If $p:M \to N$ is a Boothby-Wang fibration, then $H(N) $ $\cong H(M, \CF)$ via $p^*$. If a compact Lie group $G$ acts on $M$, preserving $\alpha$, then the $G$-action descends to $N$ and we have $H_G(M, \CF) \cong H_G(N)$ via $p^*$. For any basic form $p^*\eta \in \Omega(M, \CF)$, fiberwise integration yields
  \[ \int_M \alpha \wedge p^*\eta = 2\pi/\tau \int_N \eta. \]
\end{proposition}

It now follows from Theorem \ref{thm-localization} and Proposition \ref{prop-boothby-identification} that we recover the standard localization theorem \cite{AtiyahBottMomentMap} for integral symplectic manifolds.

\begin{theorem} Suppose that $N$ is a symplectic manifold with a Hamiltonian action of the torus $G$ and suppose furthermore that the symplectic form on $N$ is integral and that the $G$-action lifts to the $S^1$-bundle $(M, \alpha)$ in the Boothby-Wang fibration $p:M\to N$, preserving $\alpha$. Then for any $\eta \in H_G(N)$, with $e_G(\nu F)$ denoting the (ordinary) equivariant Euler class of a connected component $F\subset N^G$, we have
 \[ \int_N \eta = \sum_{F \subseteq N^G} \int_F \frac{i_F^\ast \eta}{e_G(\nu F)}. \]
\end{theorem}
\proof
Note that $\Crit \mu/S^1$ is exactly the fixed point set $N^G$. Denote by $F_j$ the connected component of $N^G$ that is obtained as $C_j /S^1$. 
It is $p^*\nu F_j \simeq \nu C_j$ and if $\theta$ is a $G$-invariant connection form on the bundle of oriented orthonormal frames of $\nu F_j$, then $\bar p^* \theta$ is a basic $G$-invariant connection form on the bundle of oriented orthonormal frames of $p^* \nu F_j$, where $\bar p (x,v):= v$. 
The Weil homomorphism is compatible with pullback such that we obtain $p^* e_G(\nu F_j)=e_G(\nu C_j,\CF)$, where the right hand side denotes the equivariant basic Euler class of $\nu C_j$.  
Applying Theorem \ref{thm-localization} and Proposition \ref{prop-boothby-identification}, we have
\allowdisplaybreaks
\begin{align*}
  \int_N \eta =\tau/(2\pi)\int_M \alpha \wedge p^*\eta =\tau/(2\pi) \sum_{C_j \subseteq \Crit \mu} \int_{C_j} \frac{i_j^\ast (\alpha \wedge \eta)}{e_G(\nu C_j,\CF)} = \sum_{F \subseteq N^G} \int_{F} \frac{i_F^\ast \eta}{e_G(\nu F)}. \qed
\end{align*}

Suppose that $0$ is a regular value of the contact moment map $\mu$. Then $0$ is also a regular value of the symplectic moment map $\tilde \mu$ that pulls back to $-\mu$ and vice versa. Denote by $M_0$ and $N_0$ the contact and symplectic quotients, respectively. We have the commutative diagram
\[ \begin{array}{ccc}
  H_G(M, \CF) & \stackrel{\cong}{\to} & H_G(N) \\
  \downarrow & & \downarrow \\
  H(M_0, \CF_0) & \stackrel{\cong}{\to} & H(N_0)
\end{array} \]

With these identifications, in exactly the same manner as the proof of the previous theorem, we also recover the usual Jeffrey-Kirwan residue theorem 
\cite{JeffreyKirwan95, JeffreyKirwan96}.

\begin{theorem} Suppose that $N$ is a symplectic manifold with a Hamiltonian action of a torus $G$. 
Suppose furthermore that the symplectic form on $N$ is integral and that the $G$-action lifts to the $S^1$-bundle $(M, \alpha)$ in the Boothby-Wang fibration $p:M\to N$, preserving $\alpha$. Let $\tilde \mu$ denote the symplectic moment map that pulls back to $-\mu$ and assume that 0 is a regular value of $\tilde \mu$. Denote the induced symplectic form on the symplectic quotient $N_0$ by $\omega_0$. For any $\eta \in H_G(N)$, we denote its image under the Kirwan map by $\eta_0$. We have
  \[ \int_{N_0} \eta_0  \wedge e^{id\alpha_0} = \frac{n_0}{\vol(G)}
    \jkres\left( \sum_{F \subseteq N^G} e^{i\ev{\tilde \mu(F), \phi}} \int_{F} \frac{i_F^\ast \eta(\phi)  \wedge e^{i\omega}}{e_G(\nu F)} [d\phi] \right). \]
\end{theorem}

\begin{remark}
	Note that we obtain the residue formula as stated in \cite{JeffreyKirwan95, JeffreyKirwan96}, without the sign that was added in \cite{JK98intersection} due to an error in \cite[Section~5]{JeffreyKirwan95}. 
	 The situation in \cite[Section~5]{JeffreyKirwan95} - in the therein defined notation - describes as follows. 
	 The only term from $e^{idz'(\theta)}$ that contributes to the integral is $(idz'(\theta))^s/s!=i^s (-1)^{s(s+1)/2} \Omega \wedge [dz']$, which causes a sign to appear in the computation. 
	 The integral is taken over a neighborhood $\CO$ of $\mu^{-1}(0)$, which is canonically oriented via the symplectic form $q^* \omega_0 + d(z'(\theta))$. 
	 The integral is computed by first taking the integral in $\fk^*$-direction, oriented via $[dz']$, followed by fiberwise integration on $\mu^{-1}(0)$, where the fibers are oriented via $\Omega$. 
	 An integral over the symplectic quotient $\CM_X$ remains; $\CM_X$ is canonically oriented via $\omega_0$. The product of these orientations differs from the canonical orientation on $\CO$ by a factor $(-1)^{s(s+1)/2}$. 
	 Hence, taking into account this change of orientation removes the additional sign (cf. also the proof of Proposition \ref{prop-q0}). For this reason, the formula as stated in \cite{JeffreyKirwan95, JeffreyKirwan96} is the correct formula to consider.
\end{remark}

\subsection{Weighted Sasakian Structures on Odd Spheres}
	For $n \geq 1$ and $w \in \R^{n+1}$, $w_j>0$, consider the sphere 
	\[S^{2n+1}= \left\lbrace z=(z_0, ..., z_n) \in \C^{n+1} \mid \sum\nolimits_{j=0}^n |z_j|^2 = 1 \right\rbrace \subset \C^{n+1},\]
	endowed with the following contact form $\alpha_w$ and corresponding Reeb vector field $R_w$
	\[\alpha_w =\frac{\tfrac{i}{2} \left( \sum_{j=0}^n z_j d\bar z_j - \bar z_j d z_j \right)}{\sum_{j=0}^n w_j|z_j|^2} , \quad R_w = i\left( \sum_{j=0}^{n} w_j(z_j \tfrac{\partial}{\partial z_j} - \bar z_j \tfrac{\partial}{\partial \bar z_j}) \right).\]
	$(S^{2n+1}, \alpha_w)$ is called a \emph{weighted Sasakian structure on $S^{2n+1}$}, cf.~\cite[Example~7.1.12]{boyer2008sasakian}. In particular, $(M, \alpha)=(S^{2n+1}, \alpha_w)$ with the metric induced by the embedding $M \hookrightarrow \mathbb{C}^{n+1}$  is a $K$-contact manifold. 
	For $w=(1, ..., 1)$, we obtain the standard contact form on the sphere. Notice that the underlying contact \emph{structure} $\ker \alpha_w$ is independent of the choice of weight $w$.
	The flow of $R_w$ is given by $\phi_t (z) = (e^{itw_0}z_0, ..., e^{itw_n}z_{n})$. 	
	Furthermore, let $G=S^1$ act (freely) on $S^{2n+1}$ with weights $\beta = (\beta_0,...\beta_n)\in \mathbb{Z}^{n+1}$, that is, by $\lambda \cdot z = (\lambda^{\beta_0} z_0, ...,\lambda^{\beta_n} z_n)$. 
	The fundamental vector field $X$ corresponding to $1 \in \R \simeq \mathfrak{s}^1$ is given by
	\[X(z) = i\left( \sum_{j=0}^{n} \beta_j(z_j \tfrac{\partial}{\partial z_j} - \bar z_j \tfrac{\partial}{\partial \bar z_j})\right)\]
	and we compute the contact moment map to be 
	\[ \mu (z) = \frac{ \sum_{j=0}^{n} \beta_j|z_j|^2 }{\sum_{j=0}^n w_j|z_j|^2}. \]
\begin{lemma}\label{lemma eq bas coh sphere}
	The  equivariant basic cohomology of $(M, \alpha)=(S^{2n+1}, \alpha_w)$ is given, as $(S(\fg^*)=\R[u])$-algebra, by 
	\[H_G(M, \CF) \cong \frac{\R[u,s]}{\ev{\prod_{j=0}^n (\beta_ju + w_j s)}}.\]
\end{lemma}
\begin{proof}
	To compute the equivariant basic cohomology of $(M, \alpha)$, consider the diagonal $S^1$-action on $\C^{n+1}$: $\lambda \cdot z := (\lambda z_0, ..., \lambda z_n)$. 
	This action is Hamiltonian with (symplectic) moment map $\Psi(z)=\frac{1}{2}\sum_j |z_j|^2$ and we obtain $M$ as $M= \Psi^{-1}(\frac{1}{2})$. The $G\times T$-action and, hence, $R$ can be extended to all of $\C^{n+1}$. 
	Consider the $G\times T$-invariant function $f:= ||\Psi-\frac{1}{2}||^2 $. Its critical set is $\{0\} \ \dot \cup \ M$,  and the critical values are $f(0)=1/4$, $f(M)=0$. 
	Hence, $M^+_{ \{0\} }:=f^{-1}((-\infty , f(0)+\epsilon ]) \cong \C^{n+1}$ and $M^-_{ \{0\} }:=f^{-1}((-\infty , f(0)-\epsilon ]) \cong M^+_M \cong M$. The Hessian $H$ of $f$ at 0 is given by $- \operatorname{id}$, which is non-degenerate and has Morse index $2(n+1)$. 
	For $z \in M$, the normal direction (to $M$) is spanned by $Y := \sum z_j \partial_{z_j}+ \bar z_j \partial_{\bar z_j}$ and $H_z(Y,Y)=2$, 
	which yields that $H_z$ is non-de\-gene\-rate in normal direction. It follows that $f$ is a $G \times T$-invariant Morse-Bott function.
	Note that $H_G(M, \CF) \cong H_{\fg \oplus \R R_w}(M)$ as an $H_G$-algebra (by \cite[\S~4.6]{GS99} or \cite[Proposition~3.9]{GT2016equivariant}), 
	where $H_{\fg \oplus \R R_w}(M)$ denotes the $\fg \oplus \R R_w$-e\-qui\-va\-ri\-ant cohomology of the $\fg \oplus \R R_w$-dga $\Omega^*(M)$, cf.~\cite[\S~2]{GS99}, \cite[\S~4]{GNTequivariant} or \cite[\S~3.2]{GT2016equivariant}.  
It follows from equivariant Morse-Theory with $f$ (cf.~\cite{Kirwan} and also \cite[\S~5.2]{Casselmann2016}) that we have a short exact equivariant Thom-Gysin sequence
\begin{align*}
	0 \rightarrow  H^{*-2(n+1)}_{\fg \oplus \R R_w} (\{0\}) \rightarrow H^*_{\fg \oplus \R R_w}(M^+_{ \{0\} } ) \rightarrow H^*_{\fg \oplus \R R_w}(M^-_{ \{0\} } )&\rightarrow 0\\
	0 \rightarrow  H^{*-2(n+1)}_{\fg \oplus \R R_w} (\{0\}) \rightarrow H^*_{\fg \oplus \R R_w}( \C^{n+1}) \rightarrow H^*_{\fg \oplus \R R_w}(M)&\rightarrow 0.
\end{align*}

The composition with the restriction \[ H_{\fg \oplus \R R_w}^{\ast-2(n+1)}(\{0\}) \to H_{\fg \oplus \R R_w}(\C^{n+1}) \stackrel{\cong}{\to} H_{\fg \oplus \R R_w}(\{0\}) \] is multiplication by the equivariant Euler class of the negative normal bundle to $0 \hookrightarrow \C^{n+1}$, which is computed to be equal to $(\frac{1}{2\pi})^{n+1}\prod_j (u\beta_j + sw_j)$. The short exact sequence then yields the claim.
\end{proof}

\begin{remark}
If all $w_j$ are positive integers, the Reeb vector field induces a locally free $S^1$-action on $M$ and $M/S^1$ is the weighted projective space $\BP(w)= (\C^{n+1} \setminus 0) / \sim$, where $(z_0, \dots, z_n) \sim (\lambda^{w_0} z_0, \dots, \lambda^{w_n} z_n)$ for any $\lambda \in \C^*$ (cf.~\cite[Example~7.1.12; \S~4.5]{boyer2008sasakian}). 
Then 
$H_G(\BP(w))\cong H_G(M, \CF) \cong \frac{\R[u,s]}{\ev{\prod_{j=0}^n (\beta_ju + w_j s)}}$.
\end{remark}
	
\begin{lemma}\label{lem crit mu sphere}
	Set $\lambda_j:= \frac{\beta_j}{w_j}$ and $J_j:= \{ l \in \{0,...,n\} \mid \lambda_l=\lambda_j\}$. $\Crit \mu$ consists of at most $ n+1$ components $D_j$, specified by $D_j = \{ z \in M \mid z_l = 0 \ \forall \ l \in \{0,...,n\}\setminus J_j \}$. 
	
	If the weights $\beta_j$ of the $G$-action are such that $\lambda_j\neq \lambda_l$ for every $j\neq l$, then $\Crit \mu$ consists of $n+1$ circles $C_j = \{ z \in M \mid z_l = 0 \ \forall \ l\neq j\}$, and $\mu(C_j) = \beta_j/w_j=\lambda_j$.
	Furthermore, $H_G(C_j, \CF) \cong \R[u]$, and the restriction $H_G(M, \CF) $ $\to H_G(C_j, \CF)$ is given by $s \mapsto -\beta_j u / w_j$. If we denote the inclusion $C_j \rightarrow M$ by $i_j$, then $\int_{C_j}\iota_j^*\alpha_w = \frac{2\pi}{w_j}$. 
	The equivariant basic Euler class $e_j$ of the normal bundle to $C_j$ in $M$ is given by
	\[ e_j = \left(\frac{u}{2\pi}\right)^{n} \prod_{k\neq j} (\beta_k - \beta_j w_k / w_j) .\]
\end{lemma}
\begin{proof}
	For every $z \in \cup D_j$, we have $X(z)=\lambda_j R_w(z)$, which yields $\cup D_j \subset \Crit \mu$. If $z \in M \setminus \cup D_j$, then there are $k \neq j$ such that $z_k, z_j \neq 0$ and $\lambda_k \neq \lambda_j$. 
	It follows that, for every $\lambda \in \R$, $\beta_j(z_j \tfrac{\partial}{\partial z_j} - \bar z_j \tfrac{\partial}{\partial \bar z_j})+ \beta_k(z_k \tfrac{\partial}{\partial z_k} - \bar z_k \tfrac{\partial}{\partial \bar z_k})$ $ \neq$ $\lambda \left( w_j(z_j \tfrac{\partial}{\partial z_j} - \bar z_j \tfrac{\partial}{\partial \bar z_j})+ w_k(z_k \tfrac{\partial}{\partial z_k} - \bar z_k \tfrac{\partial}{\partial \bar z_k})\right)$. 
	Since $(\tfrac{\partial}{\partial z_l}, \tfrac{\partial}{\partial \bar z_l} )_{l=0}^n$ form a basis of $T_z\C^{n+1}$, they are linearly independent at $z$, hence, $X(z)\notin \R R(z)$.
	
	Now suppose that $\lambda_j\neq \lambda_l$ for every $j\neq l$.
	On $C_j$, it is $R_w=w_j(x_j \partial_{y_j}-y_j\partial_{x_j})$ and $X=\beta_j(x_j \partial_{y_j}-y_j\partial_{x_j})=\frac{\beta_j}{w_j}R_w$. $d\alpha_w$ is a 2-form, so $\iota_j^*d\alpha_w=0$. In $H_{\fg \oplus \R R_w}(C_j)$, we compute
	\[ 0=[d_{\fg \oplus \R R_w}\alpha_w]=[d\alpha_w-\iota_X\alpha_w u - \iota_{R_w} \alpha_w s]=[-\tfrac{\beta_j}{w_j}u-s],\]
	thus obtaining the restriction map $s \mapsto -\beta_j u / w_j$.
	
	$\nu C_j = \mathrm{span} \{\partial_{x_k}, \partial_{y_k} \mid k\neq j\} = \C^{n}\times C_j$ is a trivial bundle that is the product of the line bundles $\mathrm{span}(\partial_{x_k},\partial_{y_k}) \times C_j$. 
	Denote by $\theta_j$ the canonical flat connection on the bundle of oriented orthonormal frames of  $\mathrm{span}(\partial_{x_k},\partial_{y_k}) $ $\times C_j$. The $\fg \oplus \R R_w$-equivariant Euler class of $\nu C_j$ then is 
	\begin{align*}
		e_{\fg \oplus \R R_w}(\nu C_j) &=\prod_{k\neq j}  e_{\fg \oplus \R R_w}(\mathrm{span}(\partial_{x_k},\partial_{y_k}))=\prod_{k\neq j} \mathrm{Pf}(-u\iota_X\theta_k - s\iota_{R_w}\theta_k)\\
		&=\prod_{k\neq j}\tfrac{1}{2\pi}(u\beta_k+sw_k)=\left(\tfrac{1}{2\pi}\right)^{n}\prod_{k\neq j} (u\beta_k + w_k(-\beta_j u / w_j))\\
		&= \left(\frac{u}{2\pi}\right)^{n}\prod_{k\neq j} (\beta_k -w_k\beta_j  / w_j).
	\end{align*}
	On $C_j$, we have $|z_j|^2=1$ and $z_l=0$ for $l \neq j$. Hence, we can parametrize $C_j$ up to a zero set by $z_j=e^{i\varphi}$, $\varphi \in (0,2\pi)$. 
	Then $\iota_j^*\alpha_w = \frac{\tfrac{i}{2}(z_jd\bar z_j - \bar z_j dz_j)}{w_j} = \frac{d\varphi}{w_j}$ and $\int_{C_j}\iota_j^*\alpha_w = \int_0^{2\pi} \frac{1}{w_j}d\varphi= \frac{2\pi}{w_j}$.
\end{proof}
	
	With our localization formula, we can now compute the contact volume of weighted Sasakian structures on odd spheres. 
	This result is known and can also by obtained by combining the observation of Martelli-Sparks-Yau \cite{MSY06} that the volume of a toric Sasakian manifold is related to the volume of the truncated cone over its momentum image and a formula by Lawrence \cite{La91} for the volume of a simple polytope (cf. \cite[\S~6.2]{GNTlocalization}). 
	Goertsches-Nozawa-T\"oben also computed the same result via a basic ABBV-type localization formula with respect to the transverse action of $\ft / \R R_w$, cf.~\cite[Corollary~6.1]{GNTlocalization}. 
	
\begin{proposition}
	The contact volume of $(M, \alpha)=(S^{2n+1},\alpha_w)$ is given by 
	\[ \vol(M, \alpha)=\frac{1}{2^n n!} \int_M \alpha \wedge (d\alpha)^n =\frac{2\pi^{n+1}}{n! w_0 \cdots w_n}.\]
\end{proposition}

\begin{proof}
Recall that $d_G\alpha=d\alpha -\mu u$. We insert the results of Lemma \ref{lem crit mu sphere} into our localization formula. 
Choose any weights $\beta_j$ such that $\lambda_j \neq \lambda_j$ for $j \neq l$ so that $\Crit \mu = \cup_{j=0}^n C_j$. 
Note that $C_j$ is 1-dimensional, so only the polynomial part of $d_G\alpha$ enters on the right hand side; we need a top degree form on the left hand side when integrating over $M$, so only $d\alpha$ enters.
\allowdisplaybreaks
\begin{align*}
\int_M \alpha \wedge (d\alpha)^n &=\int_M \alpha \wedge (d_G\alpha)^n = \sum_j(-\mu(C_j)u)^n \int_{C_j} \frac{\iota_j^*\alpha}{e_j}\\
&= (2\pi)^{n+1}(-1)^n\sum_j \left(\frac{\beta_j}{w_j}\right)^n \frac{1}{w_j\prod_{k\neq j} (\beta_k -w_k\beta_j  / w_j)}\\
&= \frac{(2\pi)^{n+1}(-1)^n}{w_0 \cdots w_n}\sum_j \frac{\beta_j^n }{\prod_{k\neq j} (w_k^{-1}\beta_k w_j -\beta_j)}.
\end{align*} 
The right hand side has to be independent of the $\beta_j$, so we can take the limit $\beta_0 \rightarrow \infty$. Then the $(j=0)$-summand tends to $(-1)^n$, the others vanish (cf.~\cite{GNTlocalization}). 
\end{proof}

Now, let us consider the special case of the odd sphere $M=S^{3} \subset \C^2$ with Sasakian structure determined by the weight $(w, 1)$ with $w > 0$ irrational. Let $G = S^1$ act on $M$ with weights $\beta=(-1,1)$. By Lemma \ref{lemma eq bas coh sphere}, we have $H_G(M, \CF) \cong \frac{\R[u, s]}{\ev{(ws-u)(s+u)}}$. 
We obtain from Lemma \ref{lem crit mu sphere} for this special case that the critical set is given by $\Crit\mu = C_0 \ \dot \cup \ C_1$, where $C_0 = S^1 \times \{0\}$ and $C_1 = \{0\} \times S^1$. 
The equivariant basic cohomology of the connected components is $H_G(C_j, \CF) \cong \R[u]$. Furthermore, $\mu(C_0) =- 1/w$, $\mu(C_1) = 1$, the Euler classes $e_j$ of the normal bundles to $C_j$ in $M$ are $e_0 = \frac{u}{2\pi} \left(1+\frac{1}{w}\right) $ and $e_1=  -\frac{u}{2\pi} \left(1+w\right) $ and the restrictions $\iota_j^*: H_G(M, \CF)\to H_G(C_j, \CF)$ are given by $\iota_0^*: s \mapsto u/w$ and $\iota_1^*: s \mapsto -u$.
Recall that we identified $\mathfrak{s}^1$ with $\R$. 
If $S^1$ is parametrized via the angle $\varphi$, then this identification corresponds to $\lambda \partial_\varphi \mapsto \lambda$. 
We determine a metric $g$ on $S^1$ by $g(\partial_\varphi, \partial_\varphi)=1$ so that the volume form is given by $\vol_{S^1} = d\varphi$, $\vol(S^1)=2\pi$. 
The induced inner product on $\R \simeq \mathfrak{s}^1$ is then multiplication so that the induced measures to consider on $\fg^*$ and $\fg$ are the standard measures $du$ and $d\phi$, respectively.

Let us consider the Mayer-Vietoris sequence (cf.~\cite[Proposition~6]{Casselmann2016}) of the pair $(M\setminus C_1, M\setminus C_0)$. 
Note that $M\setminus C_1$ equivariantlly retracts onto $C_0$, $M\setminus C_0$ equivariantly retracts onto $C_1$, and  $(M\setminus C_1) \cap (M\setminus C_0)$ equivariantly retracts onto $\mu^{-1}(0)$. Basic Kirwan surjectivity yields that the long exact Mayer-Vietoris sequence turns into short exact sequences
\[0\to H_G^*(M, \CF)\stackrel{\iota_0^* \oplus \iota_1^*}{\to} H_G^*(C_0,\CF)\oplus H_G^*(C_1,\CF) \to H_G^*(\mu^{-1}(0),\CF)\to 0.\]

Hence, we can write $\eta \in H_G(M, \CF)$ as $\eta^0 \oplus \eta^1$, with $\eta^j \in H_G(C_j, \CF) \cong \R[u]$. Considering the restriction maps, it becomes evident that $\eta^0 \oplus \eta^1$ lies in the image of $\iota_0^* \oplus \iota_1^*$ if and only if $\eta^0$ and $\eta^1$ have the same constant term, as polynomials in $u$.

We compute the argument of $\jkres$ in the residue formula to be
\begin{align*}
	&(2\pi)^2\left(\frac{e^{i\phi/w}\eta^0(\phi)}{\phi(1+w)} - \frac{e^{-i\phi} \eta^1(\phi)}{\phi(1+w)}\right)d\phi.
\end{align*}
 Note that for a rational function $g$ and $\lambda \in \R\setminus \{0\} $, the residue is given as (cf.~\cite[Proposition~3.4]{JeffreyKirwan96})
 \[\jkres^{\{t \in \R \mid t > 0 \}} \left( g(\phi) e^{i\lambda \phi} d\phi \right) = \begin{cases} 0& \lambda <0\\  \sum_{b \in \C}\mathrm{Res}_{z=b}\left( g(z) e^{i\lambda z} \right)& \mathrm{ else }\end{cases}.\]

Thus, we obtain
\begin{align*}
	\int_{M_0} \!\alpha_0  \wedge  \eta_0  \wedge e^{id\alpha_0} = \frac{1}{\vol G} \jkres\left( (2\pi)^2\left(\frac{e^{i\phi/w}\eta^0(\phi)}{\phi(1+w)}\right)d\phi\right)
	=\frac{1}{2\pi} \frac{(2\pi)^2 \eta^0(0)}{w+1} = \frac{2\pi \eta^0(0)}{w+1}.
\end{align*}

In particular, 
\begin{align}
	\int_{M_0} \alpha_0 \wedge  e^{id\alpha_0} =\int_{M_0} \alpha_0= \frac{2\pi}{1+w}.\label{eq ex res formula eta 1}
\end{align}

We will now compute the left hand side of Equation \eqref{eq ex res formula eta 1} to see that our formula holds.
Note that $\mu^{-1}(0) = S^{1}\left(\tfrac{1}{\sqrt{2}}\right) \times S^{1}\left(\tfrac{1}{\sqrt{2}}\right)$.
$\mu^{-1}(0)/G$ is $\{\phi_t\}$-equivariantly diffeomorphic to $S^{1}$ via $[z] \mapsto 2 z_1 z_0$, where $\phi_t$ acts on $S^{1}$ by $\phi_t(z)=e^{it(w+1)}z$. Under this identification, the projection $p:\mu^{-1}(0)\to M_0$ is given by $(z_0,z_1)\mapsto 2 z_1 z_0$. Denote the inclusion by $\iota:\mu^{-1}(0) \hookrightarrow M$. We then compute $\iota^* \alpha = \tfrac{2i}{w+1}(z_0d\bar z_0 + z_1 d\bar z_1)$. Since $p^* (\tfrac{i}{w+1} z d\bar z) $ $= \iota^* \alpha$, we obtain $\alpha_0 =\tfrac{i}{w+1} z d\bar z $.

Up to a zero set, $M_0 \simeq S^1$ is parametrized by $\Psi:(0, 2\pi) \to S^1$, $\psi \mapsto e^{i\psi}$. In this coordinate, $\alpha_0 = \tfrac{1}{w+1} d\psi$. Then $\int_{S^1} \alpha_0 = \int_0^{2\pi} \tfrac{1}{w+1} d\psi = \tfrac{2\pi}{w+1}$, which is exactly the right hand side of Equation \ref{eq ex res formula eta 1}.

\allowdisplaybreaks

\end{document}